\documentclass{amsart}
\usepackage{mathpreamble}

\newcommand{\s}{\widetilde{s}}
\newcommand{\g}{\widetilde{g}}
\newcommand{\M}{\widetilde{M}}

\newcommand{\cO}{\mathcal O}

\newcommand{\T}{\mathcal T}
\newcommand{\wt}{\widetilde}
\newcommand{\dd}{\mathrm d}
\newcommand{\Span}{\operatorname{span}}

\newcommand{\ad}{\operatorname{ad}}

\begin{document}

\title{Holomorphic functions on complete Hermitian manifolds with flat Chern connection}

\author{Duoxin Li}
\address{School of Mathematical Sciences, Huaqiao University, Quanzhou, Fujian, 362021, China.}
\email{{duoxinli52math@163.com}}
\author{Hanyu Wu}
\address{School of Mathematical Sciences, Xiamen University, Xiamen, Fujian, 361005, China.}
\email{{19020230157192@stu.xmu.edu.cn}}
\author{Bo Yang}
\thanks{The third-named author is partially supported by the National Natural Science Foundation of China under grant numbers 11801475, 12141101, and 12271451, and by the Natural Science Foundation of Fujian Province of China under grant number 2019J05012.}
\address{School of Mathematical Sciences, Xiamen University, Xiamen, Fujian, 361005, China.}
\email{{boyang@xmu.edu.cn}}

\date{11/04/2026}

\begin{abstract}
In this paper, we prove several results concerning the function theory of complete Hermitian manifolds with vanishing Chern curvature, which we refer to as complete Chern-flat manifolds.
First, we show that any complete simply connected Chern-flat manifold is elliptic in the sense of Gromov. Moreover, it is biholomorphic to a complex Lie group if the torsion has polynomial growth. We also construct examples of complete Chern-flat manifolds with torsion of exponential growth, but the underlying complex manifolds do not admit any complex Lie group structures. Next, we obtain sharp upper bounds on the dimensions of holomorphic functions of polynomial growth on complete Chern-flat manifolds, without any growth assumption on the torsion. Finally, we prove that any complete Chern-flat manifold admitting polynomial-growth holomorphic functions that form a local holomorphic coordinate system at some point is biholomorphic to complex Euclidean space. These results generalize our previous work, where we assume that the torsion has sublinear growth. Our approach combines Cauchy estimates with nilpotent analysis along the orbits of holomorphic flows generated by parallel unitary frames on such manifolds.

\end{abstract}


\maketitle

\markleft{Holomorphic functions on complete Chern-flat Hermitian manifolds}
\markright{Holomorphic functions on complete Chern-flat Hermitian manifolds}

\setcounter{tocdepth}{1}
\tableofcontents

\section{Introduction}

\subsection{Statement of the main results} 

In this work, we study the function theory and uniformization of noncompact Hermitian manifolds with vanishing Chern curvature which are \emph{complete with respect to the Riemannian connections}. For simplicity, we call such a manifold a \emph{complete noncompact Chern-flat} Hermitian manifold. A classical result of Boothby \cite{Boothby1958} states that any compact Chern-flat Hermitian manifold is holomorphically isometrically covered by a complex Lie group with a left-invariant metric. In our previous work \cite{MWY2026}, we establish a gradient estimate for holomorphic functions on a complete Hermitian manifold with nonnegative second Ricci curvature, which might be of independent interest. We refer to \cite[Corollaries~3.6 and~3.7]{MWY2026} for its applications to Liouville-type results on holomorphic functions and mappings between Hermitian manifolds. As a corollary of this gradient estimate, we observe the following generalization of Boothby's result.

\begin{corollary}[{\cite[Theorem~1.2]{MWY2026}}]\label{subline_intro}
    Let $(M,g)$ be a complete Chern-flat Hermitian manifold. Fix a point $p \in M$ and assume that the Chern torsion satisfies $|T|(q) \le C(1+d_g(q, p))^{\delta}$ for constants $0 \leq \delta<1$ and $C>0$. Then the universal cover of $(M,g)$ is holomorphically isometric to a complex Lie group with a left-invariant metric.
\end{corollary}

Note that the conclusion of Corollary \ref{subline_intro} is sharp, as there are examples of complete Chern-flat Hermitian manifolds with torsion of linear growth but not isometric to any left-invariant metric on a complex Lie group; see \cite[Proposition 5.2]{MWY2026}. It is natural to study the function theory and uniformization of complete Chern-flat Hermitian manifolds with torsion growth beyond the sublinear range; see \cite[Question~1.9]{MWY2026}. For example, we may wonder if such a manifold admits nonconstant bounded plurisubharmonic (PSH) functions.

Recall that a complex manifold $X$ is \emph{elliptic} in the sense of Gromov \cite{Gromov1989} if
$X$ admits a dominating holomorphic spray. By a holomorphic spray, we mean a triple $(E, \pi, s)$ with a holomorphic vector bundle $\pi: E\to X$ and a holomorphic map $s: E \rightarrow X$ such that $s(0_x)=x$ for each $x \in X$. Here $0$ denotes the zero section of $E$. A holomorphic spray $(E, \pi, s)$ is dominating if the differential of $s$ at $0_x$ maps the vertical subspace $E_x$ onto $T_x^{1,0}X$ for every $x \in X$. One of the fundamental results in \cite{Gromov1989} states that every elliptic manifold satisfies the Oka principle, and hence is an Oka manifold in the sense of \cite{Forstneric2009}. We refer to Forstneri\v{c} \cite{Forstneric2017} for a detailed exposition of related notions.

Our first result shows that any complete Chern-flat Hermitian manifold $(M, g)$ is holomorphically covered by an elliptic manifold. Moreover, this elliptic manifold is biholomorphic to a complex Lie group if the torsion of $(M, g)$ has polynomial growth. We remark that the completeness assumption of $(M, g)$ is essential here. Recall that a result of Boothby \cite{Boothby1958} states that any complex parallelizable manifold $M$ of dimension $n$ (i.e., there exist global holomorphic vector fields $E_1,\ldots,E_n$ that form a basis of $T_p^{1,0}M$ at every $p\in M$) admits a Chern-flat Hermitian metric.

\begin{theorem}\label{thm:polynomial-torsion}
Let $(M, g)$ be a connected complete Chern-flat Hermitian manifold of complex dimension $n$. 
Then its universal cover $\widetilde{M}$ is elliptic, and hence admits no nonconstant bounded PSH functions. Moreover, fix $p \in M$ and assume that the Chern torsion tensor of $(M, g)$ satisfies
\begin{equation}\label{eq:polynomial-torsion-growth}
  |T|(x)\leq C\bigl(1+d_g(p,x)\bigr)^\delta,\ \ \  \ x \in M,
\end{equation}
for constants $C>0$ and $\delta \geq 0$. Then $\widetilde{M}$
is biholomorphic to $\mathbb{C}^l\times G_0$ for some $0\leq l\leq n$ and some simply connected complex Lie group $G_0$. In particular, $\widetilde{M}$ is Stein.
\end{theorem}

Let $\mathcal{O}(M)$ be the ring of holomorphic functions on a complex manifold $M$. Fix a point $p$ on a complete Hermitian manifold $(M, g)$, and let $d_g(\cdot, p)$ denote the distance function with respect to $g$. Given a real number $d \geq 0$, we say that $f \in \mathcal{O}(M)$ has \emph{polynomial growth of order at most $d$} if there exists some constant $C(d, f)$ so that
\begin{equation}\label{porder_atmost}
|f(q)| \leq C (d_g(q, p)+1)^{d}, \ \ \ \ \forall q \in M.
\end{equation}
We further define the space of holomorphic functions of growth order at most $d$ by
\[
  \mathcal{O}_d(M,g)=\bigl\{f\in \mathcal{O}(M) \ \ |\ \ \ |f(q)|\leq C\,(1+d_g(q, p))^d
  \text{ for all }q\in M\bigr\}.
\]

Our next result provides a sharp upper dimension estimate for $\mathcal{O}_d(M,g)$ on a general complete Chern-flat Hermitian manifold $(M, g)$. In contrast to our previous result \cite[Theorem~1.6]{MWY2026}, no assumption on the growth of the torsion is needed.

\begin{theorem}\label{thm:dimension-rigidity}
Let $(M, g)$ be a complete Chern-flat Hermitian manifold of complex dimension $n$. For
any $d\geq 0$, let $\lfloor d\rfloor$ denote the largest integer not exceeding $d$. Then we have
\begin{equation}\label{eq:intro-dimension}
  \dim_{\mathbb{C}} \mathcal{O}_d(M,g)
  \leq \binom{n+\lfloor d\rfloor}{n}
  =\dim_{\mathbb{C}}\mathcal{O}_{\lfloor d\rfloor}(\mathbb{C}^n).
\end{equation}
For $q \in M$, we introduce the rank of the torsion tensor as follows:
\begin{equation}\label{eq:torsion-image}
  I_q(T)=\operatorname{span}_{\mathbb{C}}\{T_q(U,V)\ \ |\ \ U,V\in T_q^{1,0}M\},
  \ \ \ \ 
  \rho(q)=\dim_{\mathbb{C}}I_q(T),\ \ \ \ \rho=\max_{q \in M} \rho(q).
\end{equation}
If $d \geq 1$, then a refined upper dimension estimate holds.
\begin{equation}\label{eq:torsion-rank-refinement}
  \dim_{\mathbb{C}}\mathcal{O}_d(M,g)
  \leq
  \binom{n+\lfloor d\rfloor-1}{n}
  +\binom{n-\rho+\lfloor d\rfloor-1}{\lfloor d\rfloor}.
\end{equation}
Here and below, we use the convention that $\binom{p}{q}=0$ whenever $p<q$. Moreover, equality in \eqref{eq:intro-dimension} holds for some $d \geq 1$ if and only if $(M,g)$ is holomorphically isometric to the complex Euclidean space $\mathbb{C}^n$.
\end{theorem}

Next, we discuss a class of complete Chern-flat Hermitian manifolds which admit sufficiently many holomorphic functions of polynomial growth. We introduce the following definition.

\begin{definition}[Property (H)]\label{def_property}
A complete Hermitian manifold $(M,g)$ of complex dimension $n$ is said to satisfy \emph{Property (H)} if it admits holomorphic functions of polynomial growth $f_1, \ldots, f_n$ that form a local holomorphic coordinate system near some point $p \in M$.
\end{definition}

Our motivation for considering Property (H) comes from the study of complete K\"ahler manifolds with nonnegative holomorphic bisectional curvature ($BI \geq 0$ for brevity). Let $(X, g)$ be a complete simply connected noncompact K\"ahler manifold with $BI \geq 0$. Assume that $(X, g)$ does not split isometrically and it admits a nontrivial holomorphic function of polynomial growth. Then $(X, g)$ satisfies Property (H). This result follows from the work of Ni--Tam; see \cite[Corollary 6.2]{NT2003} for a similar statement. We also refer the reader to \cite{Liu2019, LT2020} for further progress on the uniformization of complete K\"ahler manifolds with $BI \geq 0$ and Euclidean volume growth. The following results, which generalize \cite[Corollary~1.5]{MWY2026} and answer \cite[Question~1.12]{MWY2026}, respectively, exhibit a rigidity phenomenon for complete Chern-flat Hermitian manifolds with Property (H).

\begin{theorem}\label{thm:high-dimension}
    Let $(M^n,g)$ be a connected complete Chern-flat Hermitian manifold with Property (H). Then
    $M$ is biholomorphic to $\mathbb{C}^n$ and $g$ is balanced. Moreover, there exist a biholomorphism $\Psi: \mathbb{C}^n \to M$ with $\Psi(0)=p$ and a set of global holomorphic coordinates $(z_1, \ldots, z_n)$ on $\mathbb{C}^n$ such that for any holomorphic function $f$ on $M$ with polynomial growth, $\Psi^{\ast} f$ is a polynomial of $z_1, \ldots, z_n$.
\end{theorem}

\begin{theorem}\label{thm:surface-rigidity}
Let $(M^2, g)$ be a complete Chern-flat Hermitian surface which satisfies Property (H). Then $(M, g)$ is holomorphically isometric to $\mathbb{C}^2$.
\end{theorem}

Note that the same conclusion of Theorem \ref{thm:surface-rigidity} does not hold in higher dimensions. Indeed, in \cite[Theorem 4.13]{MWY2026} we show that any simply connected complex nilpotent Lie group with a left-invariant Hermitian metric satisfies Property (H). Therefore, Theorem \ref{thm:high-dimension} is optimal in this sense.

The proofs of Theorems \ref{thm:polynomial-torsion}, \ref{thm:dimension-rigidity}, \ref{thm:high-dimension} and 
\ref{thm:surface-rigidity} rely on the global unitary parallel frame on the universal cover $(\widetilde{M}, \widetilde{g})$ provided by Boothby's result; see Lemma \ref{thm:boothby-frame}. Such a frame generates a natural complete holomorphic flow, hence producing entire curves along each tangent direction at a given point. The crucial analytic tool is a version of the Cauchy estimate along the orbit of the flow; see Proposition \ref{Cauchy} for a precise statement. Unlike the gradient estimates on complete Hermitian manifolds with nonnegative second Ricci curvature used in~\cite{MWY2026}, the Cauchy estimate is sharp on complete Chern-flat manifolds. 
Moreover, it holds without any growth assumption on the torsion tensor.

Assuming the torsion of a Chern-flat manifold $(M, g)$ has polynomial growth, the second assertion of Theorem 
\ref{thm:polynomial-torsion} shows that its universal cover $\widetilde{M}$ is biholomorphic to a product of $\mathbb{C}^l$ and a complex Lie group $G_0$. The proof hinges on constructing a holomorphic function $u$ of linear growth on $\widetilde{M}$. We proceed to show that $\widetilde{M}$ is biholomorphic to the product of $\mathbb{C}$ and the level set $\{u=0\}$. However, as the global holomorphic frame is parallel with respect to the Chern connection of $(M, g)$ rather than the Riemannian connection, we do not have a metric splitting as in the classical Cheeger--Gromoll theorem \cite{CG1971} or a holomorphic metric splitting in \cite{NT2003}. We refer to Remark \ref{rem:no-isometric-conclusion} for further discussion. The same idea of holomorphic $\mathbb{C}$-splitting also plays a crucial role in the proof of Theorem \ref{thm:high-dimension}.

We believe that the refined dimension estimate \eqref{eq:torsion-rank-refinement} may be useful for further investigations of complete Hermitian manifolds with a degenerate torsion condition. It is somewhat surprising that on a complete Chern-flat manifold, the rank of torsion defined in \eqref{eq:torsion-image}, rather than the growth of torsion, provides an upper bound on $\dim_{\mathbb{C}}\mathcal{O}_d(M,g)$. Roughly speaking, pick any $f \in \mathcal{O}_d(M,g)$ whose vanishing order at $q \in M$ is at least $ \lfloor d \rfloor$, i.e., all partial derivatives at $q$ of order $<\lfloor d \rfloor$ vanish. Using Proposition \ref{Cauchy}, we show that all possible partial derivatives of order $\lfloor d \rfloor$ lie in $\operatorname{Sym}^{\lfloor d \rfloor}(T_q^{1,0}M/ I_q(T))^*$ where $I_q(T)$ is defined in \eqref{eq:torsion-image}. With this observation, \eqref{eq:torsion-rank-refinement} follows from a dimension-counting argument. Our motivation for Theorem \ref{thm:dimension-rigidity} comes from the study of holomorphic functions of polynomial growth on complete K\"ahler manifolds with nonnegative curvature; see, in particular, \cite{Ni2004,CFYZ,Liu2016,YZ2022,Chu2025}.

\subsection{Further discussion}

In \cite{MWY2026}, we also posed a general ``metric moduli" problem on complete Chern-flat manifolds: given a simply connected complex parallelizable manifold $M$, can we characterize all complete Chern-flat Hermitian metrics on $M$? To illustrate this general problem in the lower dimensions, we state the following question on $\mathbb{C}^2$. A preliminary form of this question appeared as \cite[Question~1.10]{MWY2026}. We now make the following clarification.

\begin{question}\label{ques_intro}
Is any complete Chern-flat Hermitian metric on $\mathbb{C}^2$ holomorphically isometric to  
\begin{equation}
   \omega= \sqrt{-1} ( \varphi^1\wedge \overline{\varphi^1} + \varphi^2 \wedge \overline{\varphi^2} ) 
   \label{C2Herm_intro_2}.
\end{equation}
Here $\varphi^1=dz_1+(\lambda\,z_1+\mu)\,e^{\rho(z_1)}dz_2$, $\varphi^2=e^{\rho(z_1)}dz_2$ for some $\rho\in\mathcal{O}(\mathbb{C})$ and constants $\lambda, \mu \in \mathbb{C}$. In other words, consider the following complete Chern-flat Hermitian metric
\begin{equation}
   \omega_0= \sqrt{-1} ( dz_1\wedge d\overline{z}_1 + e^{2\operatorname{Re}(\rho(z_1))}dz_2 \wedge d\overline{z}_2 ) 
   \label{C2Herm_intro_3}.
\end{equation}
We expect that the $(\lambda, \mu)$-family of Hermitian metrics constructed from $\omega_0$ in the sense of \eqref{C2Herm_intro_2} exhausts all complete Chern-flat Hermitian metrics on $\mathbb{C}^2$.
\end{question}

We may also consider the analogue of Question \ref{ques_intro} for other noncompact complex parallelizable surfaces. For example, $M=\mathbb{C}^{2} \setminus \{0\}$ is complex homogeneous and hence elliptic. However, $M$ admits no complete Chern-flat metrics. Otherwise the corresponding unitary parallel holomorphic frame on $M$ extends to $\mathbb{C}^2$ by Hartogs' theorem. Then such a metric can be smoothly extended across the origin, and hence is bounded near the origin. This contradicts the fact that the length of any curve tending to the origin is infinite. However, the situation could be more subtle for a simply connected Stein surface $X$ which is elliptic and satisfies $K_X\simeq \mathcal{O}_X$. The smooth quadric 
\begin{equation}
    Q=\{z_1^2+z_2^2+z_3^2=1\} \subset \mathbb{C}^3   \label{def_Q}
\end{equation}
is such an example. Does $Q$ admit a complete Chern-flat Hermitian metric? If such a metric exists, its torsion cannot have polynomial growth by Theorem \ref{thm:polynomial-torsion}. We refer the reader to a related question \cite[Question 1.11]{MWY2026}.

It is tempting to expect that the Lie group part of Theorem \ref{thm:polynomial-torsion} remains valid without assuming that the torsion has polynomial growth. In \cite[Proposition 5.2]{MWY2026} we construct examples of complete Chern-flat Hermitian metrics on $\mathbb{C}^2$ such that their torsion tensors have the same growth as any prescribed transcendental entire function on $\mathbb{C}$. Similar examples also exist in higher dimensions; see \cite[Proposition 5.4]{MWY2026}. But the underlying complex manifolds of all these examples admit complex Lie group structures. However, non-Stein manifolds which admit complete Chern-flat metrics do exist. Precisely, we construct the following examples.

\begin{example}\label{new_example_exp}
   There exist complete Chern-flat Hermitian structures on the following simply-connected noncompact complex manifolds.
   \begin{enumerate}[label=(\arabic*)]
       \item  $M_1=Q \times \mathbb{C}$ where $Q$ is the affine quadric defined in \eqref{def_Q}. Note that $M_1$ does not admit any complex Lie group structure.
       \item  $M_2$ is the total space of a holomorphic fiber bundle over $Q \times Q$, where each fiber is a product of an elliptic curve and $\mathbb{C}$. Hence $M_2$ is $6$-dimensional and not holomorphically separable.
       \item  $M_3$ is the total space of a holomorphic fiber bundle over $Q \times Q \times Q$, where each fiber is a product of a $2$-dimensional toroidal group (also called a Cousin torus) and $\mathbb{C}$; see Proposition \ref{M_lambda_example_Cousin}. Hence $M_3$ is $9$-dimensional and not holomorphically convex.
   \end{enumerate}
   Moreover, the norm of the torsion tensor of $M_1$ has precisely exponential growth.
\end{example}

During the construction of Example \ref{new_example_exp}, we find that the affine quadric $Q$ admits a natural complete Hermitian metric with zero holomorphic sectional curvature, though it is not Chern-flat; see Remark \ref{omega_Q_curvature}. We point out that all manifolds in Example \ref{new_example_exp} are complex homogeneous spaces. It seems unlikely that the underlying complex manifolds of complete Chern-flat structures are always complex homogeneous. In any case, we believe that new analytic tools are needed to study complete Chern-flat manifolds whose torsion tensors grow faster than any polynomial rate. 

The paper is organized as follows. In Section~\ref{Sec2} we recall Boothby's result on the existence of a global unitary parallel frame on a simply connected Chern-flat manifold. Then we prove the Cauchy estimates  along the orbit of the flow generated by such a frame in Proposition \ref{Cauchy}. Section~\ref{Sec3} is devoted to the proof of Theorem \ref{thm:polynomial-torsion}. We further prove Theorems~\ref{thm:dimension-rigidity}, \ref{thm:high-dimension} and ~\ref{thm:surface-rigidity} in Section~\ref{Sec4}. The detailed construction of Example \ref{new_example_exp} is carried out in Section~\ref{Sec5}. In Appendix \ref{app_A}, we prove a special case of Theorem \ref{thm:high-dimension}, making use of Palais' theorem \cite{Palais1957} and a result of Matsushima \cite{Ma1961}. Throughout the paper we use the Einstein summation convention.

\section{A natural holomorphic flow on Chern-flat manifolds and its applications}\label{Sec2}

\subsection{Review of Boothby's result on Chern-flat Hermitian manifolds}

We refer to \cite[Section 2]{MWY2026} for more background on Hermitian manifolds.
In this subsection, we review Boothby's result on Chern-flat Hermitian manifolds in \cite{Boothby1958}.

Let $(M,J,g)$ be a Hermitian manifold of complex dimension $n$.  If $\{e_1,\ldots,e_n\}$ is a local frame of $T^{1,0}M$ with the dual coframe $\{\varphi^1,\ldots,\varphi^n\}$, the
fundamental form is $\omega=\sqrt{-1}\,g_{i\bar j}\,\varphi^i\wedge\overline{\varphi^j}$, where $g_{i\bar j}=g(e_i, \overline{e}_j)$. We use the Einstein summation convention throughout the paper. Recall that the Chern connection $D$ is the unique Hermitian connection ($Dg=DJ=0$) whose $(0,1)$ part is the Dolbeault operator $\overline{\partial}$.  Its connection, torsion, and curvature forms are determined by the structure equations
\begin{align}
  De_i&=\theta_i^{j}e_j,\ \ \ \ \ d\varphi^i=\varphi^j\wedge\theta_j^{\ i}+\tau^i,\label{eq:first-structure}\\
  d \theta_i^{j}&=\theta_i^{k}\wedge\theta_k^{j}+\Theta_i^{j}.
  \label{eq:second-structure}
\end{align}
For the Chern connection, $\tau^i$ has type $(2,0)$ and $\Theta_i^{j}$ has
type $(1,1)$.  

Now we may take complex conjugates in \eqref{eq:first-structure} to define $D \overline{e_i}=\overline{\theta_i^j} \overline{e_j}$. By such an extension of $D$, we may define the full torsion tensor for any vector fields $X, Y$ in $T M \otimes \mathbb{C}=T^{1, 0}M \oplus T^{0, 1}M$
\begin{equation}\label{eq:torsion-definition}
  T(X,Y)=D_XY-D_YX-[X,Y].
\end{equation}
Note that $T=\tau^{i}e_i$ when restricted to $T^{1, 0}M \otimes T^{1, 0}M$. We call $(M, J, g)$ Chern-flat if the curvature form $\Theta_i^{j}=0$ for one choice (hence for all choices) of the local frame $\{e_i\}_{i=1}^n$.

The following result is fundamental in the study of Chern-flat manifolds and will be used throughout our paper.

\begin{lemma}[{Boothby \cite[Theorems 1, 2, and 3]{Boothby1958}}]\label{thm:boothby-frame}
Any Chern-flat Hermitian manifold $(M, g)$ admits a local holomorphic unitary frame which is parallel with respect to the Chern connection. Moreover, such a frame can be made global if $(M, g)$ is simply connected. Conversely, any complex manifold $M$ which is parallelizable by a global holomorphic tangent frame $\{e\}$ admits a Hermitian metric $g$ with flat Chern curvature so that $\{e\}$ is unitary and parallel. 
\end{lemma}

If $(M, g)$ is a simply connected Chern-flat manifold, we fix a global parallel unitary frame in the sense that \begin{equation}\label{eq:parallel-frame}
  E_1,\ldots,E_n\in H^0(M,T^{1,0}M),\qquad
  DE_i=0,\qquad g(E_i,\overline E_j)=\delta_{ij}.
\end{equation}
Let $\varphi^1,\ldots,\varphi^n$ be the dual holomorphic coframe and $\tau^k=T_{ij}^k \varphi^i \wedge \varphi^j$. From \eqref{eq:first-structure} and \eqref{eq:torsion-definition}, we have 
\begin{equation}\label{eq:torsion-coefficients}
  \tau^k (E_i, E_j)=2T_{ij}^k,\ \ \ T(E_i,E_j)=2T_{ij}^k E_k=-[E_i,E_j].
\end{equation}
Note that each $T_{ij}^{k}$ is a global holomorphic function on $M$: the bracket of two
holomorphic vector fields is holomorphic, and the frame in \eqref{eq:parallel-frame} is holomorphic. Now we introduce the norm of the torsion $|T|^2 \coloneqq 2\sum_{p, q, i}|T_{pq}^i|^2$, and the gradient norm of a holomorphic function $f \in \mathcal{O}(M)$ as $|\nabla f|^2 \coloneqq \sum_{i} |E_i(f)|^2$.

\subsection{A natural holomorphic flow generated by the parallel unitary frame}

The goal of this subsection is to study the holomorphic flow generated by the parallel unitary frame. We observe that the flow is complete, hence producing an abundance of entire curves. Then we derive a Cauchy estimate along these entire curves in Proposition \ref{Cauchy}. As an application, we show that any holomorphic function of polynomial growth on a complete Chern-flat manifold must have exactly an integer growth order; see Corollary \ref{integer_growth}.

Recall that any smooth real vector field of uniformly bounded length on a 
complete Riemannian manifold is complete, i.e. it generates a one-parameter family of smooth diffeomorphisms
$\Psi_t$ for any $t \in (-\infty, +\infty)$. In the following we state a similar result for holomorphic vector fields on complete Hermitian manifolds.

\begin{lemma}\label{lem:complete-flows}
Assume that a complete noncompact Hermitian manifold $(M, g)$ admits a parallel
unitary holomorphic frame $\{E_i\}_{i=1}^n$ defined in \eqref{eq:parallel-frame}.
Consider
\(
V=\sum_{j=1}^n a_jE_j
\) where $a_1, \ldots, a_n \in \mathbb{C}$ are constants. Then $V$
generates a complete holomorphic flow
\[
  \Phi^V:\mathbb{C}\times M\longrightarrow M,\qquad
  (\zeta,x)\longmapsto\Phi^V_\zeta(x). 
\]
Moreover, there is a constant $C_V>0$ which only depends on $|V|_g$ such that
\begin{equation}\label{eq:flow-distance}
  d_g\bigl(x,\Phi^V_\zeta(x)\bigr)\leq C_V |\zeta|,
  \qquad x\in M,\ \zeta \in \mathbb{C}.
\end{equation}
\end{lemma}

\begin{proof}[Proof of Lemma~\ref{lem:complete-flows}]
Let $V=\frac{1}{2}(A-\sqrt{-1}B)$ where $B=JA$. As $V$ has constant length with respect to $g$, so do $A$ and $B$. So they generate complete real flows $\Psi_s^A$ and $\Psi_t^B$ for any $s, t \in \mathbb{R}$. Note that they commute as $[A,B]=0$. It is a standard fact that 
\begin{equation}\label{eq:complex-flow-definition}
  \Phi^V_{s+\sqrt{-1} t}=\Psi_t^B \circ \Psi_s^A,\ \ \ \text{for any\ }\zeta=s+\sqrt{-1} t \in \mathbb{C}.
\end{equation}
Obviously, we have
\[
  d_g\bigl(x,\Phi^V_{\zeta}(x)\bigr)  \leq \int_0^{1} \Big |\frac{d \Phi^V(s\zeta, x)}{ds}\Big|_g ds  \leq \sqrt{2} |V|_g |\zeta|.
\]
\end{proof}

\begin{remark}
On a general noncompact manifold $M$, a holomorphic vector field $V=\frac{1}{2}(A-\sqrt{-1}B)$ is said to be $\mathbb{R}$-complete if 
\begin{equation}
\frac{d}{d\zeta}\Phi^{V}(\zeta, x)=V(\Phi^{V}(\zeta, x)),\ \ \ \Phi^V(0, x)=x, \ \ \ x \in M   \label{ODE_flow}  
\end{equation} can be solved for any $\zeta \in \mathbb{R}$. $V$ is called $\mathbb{C}$-complete if
\eqref{ODE_flow}  is solvable for any $\zeta \in \mathbb{C}$. Obviously, $V$ is $\mathbb{R}$-complete if $A$ is a complete real vector field, and it is $\mathbb{C}$-complete if both $A$ and $B$ are complete real vector fields. We refer to \cite{Forstneric1996} for sufficient conditions on $M$ such that any $\mathbb{R}$-complete holomorphic vector field is $\mathbb{C}$-complete. 
\end{remark}

\begin{remark}
Let $M^n$ be a complex manifold which is parallelizable by a global frame of complete holomorphic vector fields $\{E_i\}_{i=1}^n$. Then $M$ does not necessarily admit a complete Chern-flat Hermitian metric $g$ such that $\{E_i\}_{i=1}^n$ is a global unitary frame. For example, consider $E_1=e^{z_2}\frac{\partial}{\partial z_1}$ and $E_2=e^{z_1}\frac{\partial}{\partial z_2}$ on $\mathbb{C}^2$. The corresponding metric $\omega=\sqrt{-1} (e^{-2\operatorname{Re}(z_2)} dz_1 \wedge d\overline{z_1}+e^{-2\operatorname{Re}(z_1)} dz_2 \wedge d\overline{z_2})$ is not complete on $\mathbb{C}^2$.
\end{remark}

\begin{proposition}\label{Cauchy}
    Let $(M,g)$ be a complete simply connected Chern-flat Hermitian manifold. Choose
    a parallel unitary holomorphic frame $\{E_i\}_{i=1}^n$ defined in \eqref{eq:parallel-frame}.
    Let $V$ be as in Lemma~\ref{lem:complete-flows}. Then for any $f \in \mathcal{O}_d(M,g)$ where $d \ge 1$, we have $V(f) \in \mathcal{O}_{d-1}(M,g)$. Moreover, if $f \in \mathcal{O}(M)$ satisfies 
    \begin{equation}
    \limsup_{x \rightarrow \infty} \frac{|f(x)|}{d_g(p, x)}=0,  \label{true_sublinear}
    \end{equation}
    then $f$ must be constant.
\end{proposition}

\begin{proof}[Proof of Proposition \ref{Cauchy}]
    For any fixed $x \in M$, $h(z)=f(\Phi^V_{z}(x))$ is a holomorphic function on $\mathbb{C}$. Suppose $|z|\leq R$, and then
\eqref{eq:flow-distance} gives 
\[
  d_g(p,\Phi^V_z(x))
  \leq d_g(p,x)+C_VR.
\]
It follows from the standard Cauchy estimate that
\begin{equation}\label{hgrow}
    |h'(0)|\le \frac{\sup_{|z|\le R} |h(z)|}{R} \le \frac{\sup_{|z|\le R} C(1+d_g(p,\Phi^V_z(x)))^d}{R} \le \frac{C(1+d_g(p,x)+C_VR)^d}{R}.
\end{equation}
Note that $V(f)(x)=h'(0)$. Take $R=1+d_g(p,x)$ and \eqref{hgrow} implies that $V(f) \in \mathcal{O}_{d-1}(M,g)$.

If $f$ satisfies \eqref{true_sublinear}, we rewrite \eqref{hgrow} as follows.
\begin{equation*}
    |h'(0)|\le \frac{\sup_{|z|\le R} |h(z)|}{R} \le \frac{\sup_{d_g(p,y)\le d_g(p,x)+C_V R} |f(y)|}{R}.
\end{equation*}
As $R \to \infty$, we get $V(f)(x)=0$. Hence $f$ must be constant. 
\end{proof}

\begin{corollary}\label{integer_growth}
    Let $(M,g)$ be a complete simply connected Chern-flat Hermitian manifold. Then $\mathcal{O}_d (M,g)=\mathcal{O}_{\lfloor d \rfloor} (M,g)$ for any real number $d \geq 0$.
\end{corollary}
\begin{proof}[Proof of Corollary \ref{integer_growth}]

We argue by an induction on the integer $d_0 \geq 0$ to prove the following claim.

\noindent \textbf{Claim}. For any $0 \leq d <d_0$, we pick $f \in \mathcal{O}_d(M, g)$, i.e.
there exists some $C_1=C(d, f)$ such that $|f(x)| \leq C_1 (1+d_g(x, p))^d$ holds for any $x \in M$. Then there exists $C_2=C(C_1, d, n)$ such that
\[
|f(x)| \leq C_2 (1+d_g(x, p))^{\lfloor d \rfloor},\ \ \ \text{for all } x \in M.
\]

We observe that the claim holds when $d_0=1$. By the second part of Proposition \ref{Cauchy}, $f$ is constant and we may choose $C_2=C_1$.

Assume that the claim holds for $d_0$. We consider the case $d_0<d <d_0+1$. Let $f \in \mathcal{O}_d(M, g)$ satisfy $|f(x)| \leq C_3 (1+d_g(x, p))^d$ for some constant $C_3$. By the proof of Proposition \ref{Cauchy}, it follows from \eqref{hgrow} that there exists a uniform constant $C_4=C(C_3, d)$ such that
\begin{equation}\label{uniform_est}
|V(f)(x)| \leq C_4(1+d_g(p, x))^{d-1},\ \ \ \text{for any }x \in M\ \text{and}\ |V|_g \leq 1.    
\end{equation}
Let $\gamma(t): [0,t_0] \to M$ be a unit-speed minimizing geodesic from $p$ to $x$. We may write the $(1, 0)$ part of $\gamma'(t)$ as $\sum_{i=1}^n \alpha_i(t)E_i$. Though $\alpha_i(t)$ might not be constant, we know that $\sum_{i=1}^n |\alpha_i(t)|^2=\frac{1}{2}$. It follows from \eqref{uniform_est} and the induction hypothesis that
\[
|f(x)| \leq |f(p)|+\int_0^{t_0}  \sum_{i=1}^{n} |\alpha_i(t) E_i(f)| dt \leq C_3+ \int_0^{t_0} n\,C_5 (1+t)^{d_0-1} dt \leq  C_6(1+d_g(p, x))^{d_0}.
\]
We remark that both $C_5$ and $C_6$ only depend on $C_3, d_0$, and $n$.

\end{proof}

\section{A uniformization result on Chern-flat Hermitian manifolds}\label{Sec3}

We prove Theorem~\ref{thm:polynomial-torsion} in this section. We divide the proof into the following parts. 

\begin{proposition}\label{no_torsion}
Let $(M,g)$ be a connected complete Chern-flat Hermitian manifold. Then its universal cover $\M$ is elliptic.
\end{proposition}

\begin{proof}[Proof of Proposition \ref{no_torsion}]
Let $\pi:(\M,\g) \to (M,g)$ be the universal covering with the pullback metric.  By Lemma~\ref{thm:boothby-frame}, $\M$ admits a global unitary parallel holomorphic frame $\{E_1, \ldots ,E_n\}$. Define $s:\widetilde M\times\mathbb C^n \to \M$ by 
\[
  s(x,t_1,\ldots,t_n)
  =\Phi^{E_n}_{t_n}\circ\cdots\circ\Phi^{E_1}_{t_1}(x).
\]
Here the flow $\Phi^{E_j}_{t_j}$ was defined in Lemma~\ref{lem:complete-flows}. Then $s$ is holomorphic and $s(x,0,\ldots,0)=x$. Fix $j\in\{1,\ldots,n\}$. If all parameters except $t_j$ vanish, then
\(
s(x,0,\ldots,t_j,\ldots,0)=\Phi^{E_j}_{t_j}(x).
\)
By the definition of the flow,
\[
\frac{\partial}{\partial t_j}
s(x,0,\ldots,t_j,\ldots,0)|_{t_j=0}=E_j(x).
\]
Therefore, the differential $ds$ maps each of the $\mathbb{C}^n$-directions to $E_1(x),\ldots,E_n(x)$ respectively. Hence $s$ defines a dominating holomorphic spray on the trivial bundle $\widetilde M\times\mathbb C^n\to\widetilde M$. Consequently, $\widetilde M$ is elliptic in the sense of Gromov, hence it is an Oka manifold; see \cite[0.6, p.~855]{Gromov1989} and \cite[Corollary~5.6.14, p.~230]{Forstneric2017}.
\end{proof}

Now we discuss the structure of the universal cover $\widetilde{M}$ in Proposition \ref{no_torsion} under the additional assumption that the torsion has polynomial growth. The key observation is the following splitting lemma. 
\begin{lemma}\label{split}
    Let $(M^n,g)$ be a complete simply connected Chern-flat manifold. Assume that there exists a non-constant holomorphic function $u \in \mathcal{O}_1(M,g)$. Then $M$ is biholomorphic to $\mathbb{C} \times N$ where $N=u^{-1}(0)$. Moreover, $(N,g|_N)$ is also a simply connected complete Chern-flat manifold.
\end{lemma}

\begin{proof}[Proof of Lemma \ref{split}]
    Let $\{E_1,\ldots,E_n\}$ be a global holomorphic frame defined in \eqref{eq:parallel-frame}. By Proposition \ref{Cauchy}, there exist constants $c_1,\ldots,c_n$ which are not all zero such that $E_j(u)=c_j$ for $j=1,2,\ldots,n$. Without loss of generality, assume $c_1 \ne 0$ and set $X=\frac{1}{c_1}E_1$. Then $X(u)=1$. It follows that \(\frac{d}{dt} u(\Phi^X_t(x))=1\). Hence
\begin{equation}\label{eq:translation-u}
  u(\Phi^X_t(x))=u(x)+t,\ \ \ \ t \in \mathbb{C}.
\end{equation}
In particular, $N=u^{-1}(0)$ is nonempty and thus a closed submanifold of $M$. We define 
\begin{equation}\label{biholo_F}
    F:\mathbb{C}\times N \longrightarrow M, 
  \qquad F(t,x)=\Phi^X_t(x).
\end{equation}
Then $F$ is holomorphic. Obviously, $F$ has an inverse map
\[
F^{-1}(x)=\bigl(u(x),\Phi^X_{-u(x)}(x)\bigr).
\]
Thus $M$ is biholomorphic to $\mathbb{C} \times N$.  Since $M$ is simply connected, $N$ is simply connected as well. $(N,g|_N)$ is complete since it is closed in $M$.

It remains to show that $(N,g|_N)$ is Chern-flat. Note that $du=\sum c_j \varphi^j$. Under a suitable unitary transformation, we may choose a new frame $\{\widetilde{E}_i\}$ with its dual coframe $\{\widetilde{\varphi}^i\}_{i=1}^n$ such that $du=c \widetilde{\varphi}^1$. Then $\widetilde{E}_2, \ldots,\widetilde{E}_n$ are contained in $\ker du$, whose restrictions to $N$ define a global holomorphic unitary frame of $g|_N$. By Lemma \ref{thm:boothby-frame}, $g|_N$ is also Chern-flat.
\end{proof}

\begin{proof}[Proof of Theorem~\ref{thm:polynomial-torsion}]

Let $\pi: (\widetilde{M}, \widetilde{g}) \to (M,g)$ be the universal covering map. If $\widetilde{p}$ is a lift of $p$, then $d_g(p,\pi(x))\leq d_{\g}(\widetilde{p},x)$. Thus the torsion of $(\widetilde{M}, \widetilde{g})$ also satisfies the same bound \eqref{eq:polynomial-torsion-growth}. Therefore it suffices to assume that $M$ is simply connected and its torsion satisfies \eqref{eq:polynomial-torsion-growth}.

Choose a global holomorphic frame $\{E_j\}$ defined in \eqref{eq:parallel-frame}. Note that each $T_{ij}^{k}$ defined in \eqref{eq:torsion-coefficients} is a global holomorphic function. If all these $T_{ij}^{k}$ are constant, then $(M,g)$ is holomorphically isometric to a complex Lie group with a left-invariant metric by \cite{Boothby1958}. Assume that $T_{ij}^{k}$ is non-constant for some $i,j$, and $k$. Applying Proposition \ref{Cauchy} repeatedly, we can conclude that there exists a maximal integer $s \geq 1$ such that for any $i_1,\ldots,i_s \in \{1,2,\ldots,n\}$
\[
  E_{i_s}\cdots E_{i_1}(T_{ij}^k)
\]
is constant but not all of them are zero. Set
\[
  u=E_{i_{s-1}}\cdots E_{i_1}(T_{ij}^k)\ \ \text{if}\ s>1;\ \ \ \ u=T_{ij}^k\ \text{when}\ s=1.
\]
Then $u$ has exactly linear growth in the sense that $\limsup_{x \rightarrow \infty} \frac{|u(x)|}{d_g(p, x)} \in (0, \infty)$, by a similar argument as in the proof of \cite[Theorem 1.2]{MWY2026}.
It follows from Lemma \ref{split} that $M$ is biholomorphic to $\mathbb{C} \times N$. 

Recall that $\ker(du)$ has an orthonormal basis
$\widetilde{E}_2,\ldots,\widetilde{E}_n$.
The restrictions of these fields to $N$ form a global unitary holomorphic
frame. Consequently, the induced torsion $T^N$ of $g|_N$ is the restriction of the torsion $T$ to $\ker(du)$. For fixed $p_0\in N$, we have
\begin{equation}\label{distance}
    d_g(p,y)\leq d_g(p,p_0)+d_g(p_0,y)
  \leq d_g(p,p_0)+d_{g|_N}(p_0,y).
\end{equation}
Thus $T^N$ also satisfies \eqref{eq:polynomial-torsion-growth}. By the same argument as before, either the component coefficients of $T^N$ are all constants, in which case $(N,g|_N)$ is holomorphically isometric to a complex Lie group; or there exist some non-constant $(T^N)_{ij}^k$ with polynomial growth, in which case we can again apply the splitting Lemma \ref{split}. 

Repeating the above steps, we finally get the conclusion that either $M$ is biholomorphic to $\mathbb{C}^n$ or $M$ is biholomorphic to $\mathbb{C}^l \times N_l$ where $0<l<n$ and $(N_l,g|_{N_l})$ admits no nonconstant holomorphic functions with polynomial growth. In particular, the torsion of $(N_l,g|_{N_l})$ is constant. By \cite{Boothby1958} or \cite[Theorem 1.2]{MWY2026}, $(N_l,g|_{N_l})$ is holomorphically isometric to a complex Lie group $G$ with a left-invariant metric. In either case, $M$ is biholomorphic to a simply connected complex Lie group. Note that every simply connected complex Lie group is Stein, by the theorem of Matsushima--Morimoto~\cite{Ma1960,MM1960}.
\end{proof}

\begin{remark}\label{rem:no-isometric-conclusion}

We consider an example of a Chern-flat Hermitian metric which appears in \cite[Proposition 5.4]{MWY2026}. For any two polynomials $\eta$ and $\xi$ on $\mathbb{C}^2$, the following three global holomorphic $1$-forms
\begin{equation}
  \psi^1=dz_1,\ \ \  \psi^2=dz_2, \ \ \ \psi^3=dz_3-\eta(z_1, z_2)dz_1-\xi(z_1, z_2)dz_2, 
  \label{eta_xi}    
\end{equation}
define a complete Chern-flat Hermitian metric on $\mathbb{C}^3$ as $ \omega=\sqrt{-1} \sum_{k=1}^3 \psi^k \wedge \overline{\psi^k}$. If we choose $\eta=0$ and $\xi=\frac{1}{3}z_1^3$, then the only nonzero torsion component is $T_{12}^3=-T_{21}^3=-\frac{1}{2}z_1^2$. According to the proof of Theorem~\ref{thm:polynomial-torsion}, we choose $u=E_1(T_{12}^3)=-z_1$. Then the complex submanifold $(N \coloneqq \{u=0\}, \omega|_N)$ is the complex Euclidean space $\mathbb{C}^{2}$. This example shows that in general the biholomorphism in Theorem~\ref{thm:polynomial-torsion} is not an isometry.
\end{remark}

In the proof of Theorem \ref{thm:polynomial-torsion}, we observe that a simply connected complete Chern-flat manifold admits non-constant holomorphic functions with polynomial growth if and only if it admits non-constant holomorphic functions with linear growth. Thus the splitting lemma is valid if $\mathcal{O}_d(M,g)$ contains a non-constant function for some $d>0$. Indeed, we can prove the following result, which complements Theorem \ref{thm:polynomial-torsion}.

\begin{corollary}\label{cor:split_factor}
    Let $(M^n,g)$ be a simply connected complete Chern-flat manifold. Then $M$ is biholomorphic to $\mathbb{C}^k \times N$ where $0\le k \le n$ and $N$ is either a point or a simply connected complete Chern-flat Hermitian manifold. In the latter case, $N$ admits no nonconstant holomorphic functions of polynomial growth with respect to any complete Chern-flat Hermitian metric on it.
\end{corollary}

\begin{remark}\label{rem:no-poly-splitting}

Consider the complete Chern-flat metric $(M, \omega_0)$ on $\mathbb{C}^2$ defined by \eqref{C2Herm_intro_3}, which is originally from \cite[Proposition 5.2]{MWY2026}. Corollary \ref{cor:split_factor} applies even when $\rho(z)$ is a transcendental entire function on $\mathbb{C}$, although the torsion of $(M,\omega_0)$ does not satisfy the polynomial growth condition. In this case, $M$ splits as $\mathbb{C} \times \{z_1=0\}$. As the induced metric on $\{z_1=0\}$ is Euclidean, $\{z_1=0\}$ is biholomorphic to $\mathbb{C}$. The factor $N$ appearing in the conclusion of Corollary \ref{cor:split_factor} is trivial for this example.
\end{remark}

\section{Holomorphic functions of polynomial growth}\label{Sec4}

We prove Theorems \ref{thm:dimension-rigidity}, \ref{thm:high-dimension} and \ref{thm:surface-rigidity} in this section. 

\subsection{Dimension estimates on holomorphic functions of polynomial growth}

We begin by reviewing some basic facts on holomorphic jet maps; see, for example, \cite{Demailly_book} for more background. Let $M$ be a complex manifold. Fix $q \in M$. Let $\mathcal{O}_{M,q}$ denote the local ring of germs of holomorphic functions at $q$, and let
\(
\mathfrak{m}_q=\bigl\{[f]_q\in\mathcal{O}_{M,q}:f(q)=0\bigr\}
\)
be its maximal ideal. For $k\geq 1$, the $k$-th power of $\mathfrak{m}_q$ is the ideal
\[
\mathfrak{m}_q^k=\left\{\sum_{l=1}^{N}f_{l,1}\cdots f_{l,k}:N\in\mathbb{N},\ f_{l,j}\in\mathfrak{m}_q\right\}.
\]

Choose a local holomorphic coordinate chart $(z^1,\ldots,z^n)$ centered at $q$. Then a germ $[f]_q$ belongs to $\mathfrak{m}_q^k$ if and only if
\[
\partial^\alpha f(q)=0
\qquad\text{for every multi-index $\alpha$ with }|\alpha|<k.
\]
Thus, $\mathfrak{m}_q^k$ consists precisely of the germs vanishing to order at least $k$ at $q$.

The space of holomorphic $k$-jets at $q$ is defined by
\begin{equation}
J_q^k\mathcal{O}_M \coloneqq \mathcal{O}_{M,q}/\mathfrak{m}_q^{k+1}.      \label{jet_def}
\end{equation}
One can define the $k$-jet map
\begin{equation}
 j_q^k:\mathcal{O}(M)\longrightarrow J_q^k\mathcal{O}_M,\qquad j_q^k f=[f]_q+\mathfrak{m}_q^{k+1}.
 \label{jet_map_def}
\end{equation}
In local coordinates, $j_q^k f$ is determined by the derivatives $\partial^\alpha f(q)$ with $|\alpha|\leq k$. In particular, $j_q^k f=0$ if and only if $\partial^\alpha f(q)=0$ for any $|\alpha|\leq k$.

There are canonical isomorphisms
\begin{equation}
\mathfrak{m}_q^k/\mathfrak{m}_q^{k+1}\simeq \operatorname{Sym}^k(T_q^{1,0}M)^*.  \label{canonical_iso}
\end{equation}
Hence the dimension of $J_q^k\mathcal{O}_M$ can be computed.
\begin{equation}
\dim_{\mathbb{C}}J_q^k\mathcal{O}_M=\sum_{l=0}^{k}\dim_{\mathbb{C}}\operatorname{Sym}^l(T_q^{1,0}M)^*=\sum_{l=0}^{k}\binom{n+l-1}{l}=\binom{n+k}{n}.
\end{equation}

\begin{lemma}\label{upper}
    Let $(M^n,g)$ be a complete simply connected Chern-flat manifold. Then for any integer $d\ge 0$, 
    \begin{equation}\label{eq:dimension}
  \dim_{\mathbb{C}} \mathcal{O}_d(M,g)
  \leq \binom{n+d}{n}
  =\dim_{\mathbb{C}}\mathcal{O}_{d}(\mathbb{C}^n).
\end{equation}
\end{lemma}

\begin{proof}[Proof of Lemma \ref{upper}]
    Consider the holomorphic $d$-jet map
\[
  j_q^d:\mathcal{O}_d(M,g)\longrightarrow J_q^d\mathcal{O}_M.
\]
It suffices to show that $j_q^d$ is injective. Suppose $j_q^d f=0$. Choose the global
parallel frame $E_1,\ldots,E_n$ on $M$ defined in \eqref{eq:parallel-frame}. Let $(z^1,\ldots,z^n)$ be local holomorphic coordinates near $q$. Then all derivatives of $f$ up to order $d$ vanish at $q$. Next we write
\(
  E_i=a_i^{\alpha}(z)\frac{\partial}{\partial z^\alpha}
\)
in a neighborhood of $q$. For any integer $k \geq d+1$, repeatedly applying Proposition \ref{Cauchy}, we conclude that 
\begin{equation}\label{high_vanish}
    E_{i_1}\cdots E_{i_k}(f)(q)=0
\end{equation}
for all $i_1,\ldots,i_k$ belonging to $\{1,2,\ldots,n\}$. Now assuming $k=d+1$, we expand \eqref{high_vanish} in local coordinates. Since all derivatives of $f$ up to order $d$ vanish, we get the top-order term 
\[
  a_{i_1}^{\alpha_1}(q)\cdots
  a_{i_k}^{\alpha_k}(q)
  \frac{\partial^k f}
  {\partial z^{\alpha_1}\cdots\partial z^{\alpha_k}}(q)=0.
\]
Note that $(a_{i}^{\alpha})$ is invertible in a neighborhood of $q$ and we can take $(a_{i}^{\alpha}(q))$ as the identity matrix. It follows that all partial derivatives of $f$ (with respect to $z^i$) vanish at $q$ up to order $d+1$. By an induction on $k$, $f$ is identically zero and hence the jet map $j_q^d$ is injective. Therefore
\[
  \dim_{\mathbb{C}}\mathcal{O}_d(M,g) \leq \dim_{\mathbb{C}}J_q^d\mathcal{O}_M =\binom{n+d}{n}.
\]
\end{proof}

We now turn to a refined dimension estimate of $\mathcal{O}_d(M,g)$ using the rank of the torsion tensor $\rho$ defined in \eqref{eq:torsion-image}.

\begin{lemma}\label{sharp_upper}
    Let $(M^n,g)$ be a complete simply connected Chern-flat manifold. Then for any integer $d\ge 1$, \begin{equation}\label{torsion_bound}
  \dim_{\mathbb{C}}\mathcal{O}_d(M,g)
  \leq
  \binom{n+d-1}{n}
  +\binom{n-\rho+d-1}{d},
\end{equation} where $\rho$ is the rank of the torsion tensor defined in \eqref{eq:torsion-image}. 
\end{lemma}

\begin{proof}[Proof of Lemma \ref{sharp_upper}]
    For an integer $d\ge 1$, set
\[
  K_q=\ker\bigl(j_q^{d-1}:\mathcal{O}_d(M,g)\to J_q^{d-1}\mathcal{O}_M\bigr).
\]
Then we have
\begin{equation}\label{eq:jet-splitting-bound}
  \dim_{\mathbb{C}} \mathcal{O}_d(M,g)
  \le \dim_{\mathbb{C}} J_q^{d-1}\mathcal{O}_M+\dim_{\mathbb{C}} K_q.
\end{equation}

By the canonical isomorphism \eqref{canonical_iso}, we have the following exact sequence.
\begin{equation}
    0\to \operatorname{Sym}^d(T_q^{1,0}M)^* \to J_q^d\mathcal{O}_M \to J^{d-1} _q\mathcal{O}_M \to 0.
\end{equation}
We consider the linear map $\mathcal{S}: \mathcal{O}_d(M,g) \to \operatorname{Sym}^d(T_q^{1,0}M)^*$ defined by
\[
\mathcal{S}(f)\left(\frac{\partial}{\partial z_{i_1}}\Bigg|_q \otimes \frac{\partial}{\partial z_{i_2}}\Bigg|_q \otimes \cdots \otimes \frac{\partial}{\partial z_{i_d}}\Bigg|_q\right)=\frac{\partial^d f}{\partial z_{i_1} \partial z_{i_2}\cdots \partial z_{i_d}}(q).
\]
By Lemma \ref{upper}, the $d$-jet map $j_q^d$ is injective. It follows that the restriction of $\mathcal{S}$ to $K_q$ is also injective. For $f\in K_q$, we get 
\begin{equation}
    \mathcal{S}(f)(E_{i_1}(q) \otimes E_{i_2}(q) \otimes \cdots \otimes E_{i_d}(q))=E_{i_1} E_{i_2} \cdots E_{i_d}f(q).
\end{equation}

Meanwhile, for any $i_1,i_2,\ldots,i_{d-1},i,j$ belonging to $\{1,2,\ldots,n\}$, we have
\begin{equation*}
    E_{i_1} E_{i_2} \cdots E_{i_{d-1}}[E_i,E_j]f=0
\end{equation*}
by Proposition \ref{Cauchy}. Consequently,
\[
\mathcal{S}(f)(E_{i_1}(q) \otimes E_{i_2}(q) \otimes \cdots \otimes E_{i_{d-1}}(q) \otimes [E_i,E_j](q))=0.
\]
Note that the vectors $[E_i,E_j](q)$ where $1 \leq i, j \leq n$ span $I_q(T)$ by \eqref{eq:torsion-coefficients}. Thus $\mathcal{S}(f) \in \operatorname{Sym}^d(T_q^{1,0}M/ I_q(T))^*$ for all $f \in K_q$. Hence
\[
\dim_{\mathbb{C}} K_q \le \dim_{\mathbb{C}} \operatorname{Sym}^d(T_q^{1,0}M/ I_q(T))^*=\binom{n-\rho(q)+d-1}{d}.
\]
From \eqref{eq:jet-splitting-bound} and the fact $\dim_{\mathbb{C}} J_q^{d-1}\mathcal{O}_M= \binom{n+d-1}{n}$, we have 
\(
\dim_{\mathbb{C}}\mathcal{O}_d(M,g)
  \leq
  \binom{n+d-1}{n}
  +\binom{n-\rho(q)+d-1}{d}.
\)
The estimate \eqref{torsion_bound} follows by minimizing the right-hand side over $q\in M$.

\end{proof}

\begin{proof}[Proof of Theorem~\ref{thm:dimension-rigidity}]
Let $(\M,\g)$ be the universal cover of $(M,g)$. Note that $\dim_{\mathbb{C}} \mathcal{O}_d(M,g) \le \dim_{\mathbb{C}} \mathcal{O}_d(\M,\g)$. Therefore, it suffices to assume that $M$ is simply connected. Thanks to Corollary \ref{integer_growth}, we may assume that $d$ is a positive integer. Hence \eqref{eq:intro-dimension} and \eqref{eq:torsion-rank-refinement} follow directly from Lemmas \ref{upper} and \ref{sharp_upper}, respectively.

Now assume that the equality holds in \eqref{eq:intro-dimension} for some $d \geq 1$. If the torsion tensor
$T$ is nonzero, then $\rho(q) \geq 1$ at some point $q \in M$. It follows from \eqref{eq:torsion-rank-refinement} that 
\[
  \dim_{\mathbb{C}}\mathcal{O}_d(M,g)
  \leq \binom{n+\lfloor d\rfloor-1}{n}+\binom{n+\lfloor d\rfloor-2}{\lfloor d\rfloor}
  <\binom{n+\lfloor d\rfloor}{n},
\]
which is a contradiction.  Thus $T\equiv0$ and $g$ is K\"ahler. Hence $(M,g)$ is holomorphically isometric to the complex Euclidean space $\mathbb{C}^n$. 

Assume that the equality holds in \eqref{eq:intro-dimension} for some $d \geq 1$ and $(M, g)$ is not necessarily simply connected. It suffices to recall the following fact: if $(M, g)$ is any nontrivial isometric quotient manifold of the complex Euclidean space $\mathbb{C}^n$, the strict inequality $\dim_{\mathbb{C}}\mathcal{O}_{d} (M, g)< \dim_{\mathbb{C}} \mathcal{O}_{\lfloor d \rfloor}(\mathbb{C}^n)$ holds for any $d \geq 1$.
\end{proof}

By the proof of Lemma \ref{sharp_upper}, we have the following observation; see \cite[Theorem~1.4]{MWY2026} for a related result.

\begin{corollary}\label{semisimple}
    Let $(M^n, g)$ be a complete Chern-flat Hermitian manifold. Suppose the rank of its torsion tensor $\rho=n$. Then $(M^n,g)$ admits no non-constant holomorphic functions with polynomial growth. In particular, any semi-simple complex Lie group admits no non-constant holomorphic functions with polynomial growth with respect to a left-invariant Hermitian metric.
\end{corollary}

\subsection{Chern-flat manifolds which satisfy Property (H)}

In this subsection, we study complete Chern-flat manifolds with Property (H); see Definition \ref{def_property}. The main goal is to prove Theorems \ref{thm:high-dimension} and \ref{thm:surface-rigidity}. First, we give a proof of Theorem \ref{thm:high-dimension} using Lemma \ref{split}, except that the proof that $g$ is balanced is postponed to the end of this subsection.

\begin{proof}[Proof of Theorem \ref{thm:high-dimension}: Part I]
    We argue by the following steps.

    \medskip

    \textbf{Step 1:}  We assume that $M$ is simply connected. 

    \medskip
    
   \textbf{Step 1a:} We show that $M$ is biholomorphic to $\mathbb{C}^n$. 

    \medskip
   
   By the same argument in the proof of Theorem \ref{thm:polynomial-torsion}, we use Proposition  \ref{Cauchy} repeatedly and conclude that there exists a maximal integer $s \geq 1$ such that for any $i_1,\ldots,i_s \in \{1,2,\ldots,n\}$,
\[
  E_{i_s}\cdots E_{i_1}(f_1)
\]
is constant but not all of them are zero. Set
\[
  u=E_{i_{s-1}}\cdots E_{i_1}(f_1)\ \ \text{if}\ s>1;\ \ \ \ u=f_1\ \text{when}\ s=1.
\] 
It follows from Lemma \ref{split} that $M$ is biholomorphic to $\mathbb{C} \times N$. We may assume that $u(p)=0$, and then $p \in u^{-1}(0)$. By \eqref{distance}, the restrictions of $f_1,f_2,\ldots,f_n$ to $N$ are also of polynomial growth with respect to the induced complete Chern-flat metric on $N$. Since $f_1,f_2,\ldots,f_n$ give local coordinates on $M$, they cannot all be constant on $N$. Hence $N$ admits a non-constant holomorphic function with polynomial growth. Lemma \ref{split} can also be applied to $N$. Moreover, $N$ also satisfies the Property (H) with respect to the induced metric. It follows by induction that $M$ is biholomorphic to $\mathbb{C}^n$. 

   \medskip
   
  \textbf{Step 1b:} We prove the existence of $\Psi$ by an induction on the dimension $n$.

 \medskip
  
  For $n=1$, $M$ is holomorphically isometric to the complex plane $\mathbb{C}$ while $f_1$ is a polynomial. This is trivial.  Suppose that for any simply connected complete Chern-flat manifold $N^{n-1}$ with enough holomorphic functions of polynomial growth to give local coordinates, there exists a biholomorphism $\widehat{\Psi}: \mathbb{C}^{n-1} \to N$ that satisfies the conclusion of Theorem \ref{thm:high-dimension}. Now we consider a simply connected complete Chern-flat manifold $M^{n}$ satisfying the assumption of Theorem \ref{thm:high-dimension}. Together with Step 1, \eqref{eq:translation-u}, and \eqref{biholo_F}, this yields a biholomorphism $F: \mathbb{C} \times N \to M, \ (t,y) \mapsto \Phi^X_t(y)$. For any $f \in \mathcal{O}_d(M,g)$, we introduce
\(
h(y)= f\left( \Phi^X_t(y)\right) \in \mathcal{O}_d(N,g|_N)
\) for each $t \in \mathbb{C}$. Fix any $y \in N$. We may write
\begin{equation}\label{poly_extension}
    f\left( \Phi^X_t(y)\right)=\sum_{j=0}^d \frac{1}{j!} X^j (f)(y) t^j, \qquad \text{where} \ X^0(f)=f.
\end{equation}
Here we use that $X^{j}(f)\equiv 0$ for $j>d$ by Proposition \ref{Cauchy}. By the induction hypothesis, $y= \widehat{\Psi}(z_1, \ldots, z_{n-1})$ and each
\(
Q_j(z)= X^j (f)(y)= X^j (f)( \widehat{\Psi}(z) )
\)
is a polynomial in $z_1, \cdots, z_{n-1}$. Let $\Psi(t,z)=F(t, \widehat{\Psi}(z))=\Phi^X_t(\widehat{\Psi}(z))$. Together with \eqref{poly_extension}, we get that
\(
f\circ \Psi (t,z)=\sum_{j=0}^d \frac{1}{j!} Q_j(z) t^j,
\)
which is a polynomial in $t, z_1, \ldots, z_{n-1}$.

  \medskip

\textbf{Step 2:} We assume that $M$ is not necessarily simply connected. 

  \medskip

The above argument holds on the universal cover $\M$. Write $M \simeq \M /\Gamma$ where $\Gamma \subset \operatorname{Aut}(\widetilde{M})$ is the deck transformation group. Let $\widetilde{f}_j$ be the lift of $f_j$ for $j=1,\ldots,n$ and let $\widetilde{p}$ be a lift of $p$. Take a biholomorphism $\Psi: \mathbb{C}^n \to \M$ with $\Psi(0)=\widetilde{p}$ obtained in Step 1. Define $P_j= \widetilde{f_j} \circ \Psi$ and 
\[
P=(P_1,\ldots,P_n): \mathbb{C}^n \to \mathbb{C}^n \ \ \text{is a polynomial map with} \ \ \det P'(0)\ne 0.
\]
For simplicity, we assume that $P(0)=0$. Note that $P(\Psi^{-1} \circ \gamma \circ \Psi)=P$ for each $\gamma \in \Gamma$. As $P'(0)$ is invertible, we see that $P'(\Psi^{-1} (\gamma (\widetilde{p})))$ is also invertible for each $\gamma \in \Gamma$. By the inverse function theorem, $P$ is a local biholomorphism at each $\Psi^{-1} (\gamma (\widetilde{p}))$. Consider the affine algebraic set $P^{-1}(0)$. We observe that $\Psi^{-1} (\gamma (\widetilde{p}))$ is an isolated point for each $\gamma \in \Gamma$. As every affine algebraic set has only finitely many irreducible components, it follows that $\{\Psi^{-1} (\gamma (\widetilde{p})) \ | \ \gamma \in \Gamma\}$ is a finite set.  Moreover, if $\Psi^{-1} (\gamma_1 (\widetilde{p})) =\Psi^{-1} (\gamma_2 (\widetilde{p}))$, then $\gamma_1=\gamma_2$ by the uniqueness of the lifting. Hence $\Gamma$ is finite. Note that $\pi_1(M)$ is torsion-free; see, for example, \cite[Proposition 2.45]{Hatcher}. Then $\Gamma$ is trivial and $M$ is biholomorphic to $\mathbb{C}^n$.

It remains to show that $g$ is balanced. We postpone its proof to the end of Section \ref{Sec4}.

\end{proof}

\begin{proposition}\label{prop:nilpotent-brackets}
Let $(M^n,g)$ be a complete simply connected Chern-flat manifold with Property (H). Let $\{E_1,\ldots,E_n\}$ denote the global unitary holomorphic frame  defined in \eqref{eq:parallel-frame}.
Then the complex Lie algebra of holomorphic vector fields generated by $\{E_1,\ldots,E_n\}$ is finite-dimensional and nilpotent.
\end{proposition}

\begin{proof}[Proof of Proposition~\ref{prop:nilpotent-brackets}]
Choose $d>0$ so that all $f_1,\ldots,f_n \in \mathcal{O}_d(M,g)$. Note that any element of the complex Lie algebra generated by $\{E_1,\ldots,E_n\}$ can be expressed as a finite linear combination of
\[
E_i, \ [E_i, E_j], \ [[E_i, E_j], E_k],\ [[[E_i, E_j], E_k], E_l], \ldots,  \ \ \text{where}\ 1 \leq i, j, k, l, \ldots \leq n.
\]
By Proposition \ref{Cauchy}, any iterated bracket of length $k \geq d+1$ in the above list sends $f_1, \ldots, f_n$ to zero. As $f_1,\ldots,f_n$ give local coordinates near $p$, any such bracket vector field must vanish on a neighborhood of $p$. Therefore it vanishes identically on $M$. Consequently, the complex Lie algebra generated by $\{E_1,\ldots,E_n\}$ is finite-dimensional and nilpotent.
\end{proof}

\begin{proof}[Proof of Theorem~\ref{thm:surface-rigidity}]
According to Theorem \ref{thm:high-dimension}, it suffices to assume that $M$ is simply connected.
Suppose that $f_1$ and $f_2$ are holomorphic functions of polynomial growth and $df_1\wedge df_2$ is nonzero at some point on $M$. Let $X,Y$ be the global parallel holomorphic frame from \eqref{eq:parallel-frame} and $\alpha,\beta$ the dual coframe.
It follows from Proposition
\ref{prop:nilpotent-brackets} that there exists an integer $N$ such that
\[
  \ad_X^NY=0,
  \qquad
  \ad_Y^NX=0.
\]

Let $\Phi_t^X$ be the complete complex flow of $X$, and define $Y_t=(\Phi_{t}^X)^*(Y) \coloneqq (\Phi_{-t}^X)_*Y$. Then we may check that
\begin{equation*}
    \frac{d}{dt} Y_t=\mathcal{L}_X Y_t=[X,Y_t],\ \ \ \ \frac{d^k}{dt^k}Y_t= [X,\frac{d^{k-1}}{dt^{k-1}} Y_t]=\ad_X^k Y_t,\ \ \ \text{for any }k \ge 1.
\end{equation*}
As $\frac{d^k}{dt^k} Y_t=0$ for any $k \geq N$, we obtain a polynomial expansion of $Y_t$.
\begin{equation}\label{taylor}
    Y_t=Y+ t[X,Y]+\frac{t^2}{2}[X,[X,Y]]+\cdots+ \frac{t^{N-1}}{(N-1)!} \ad_X^{N-1}Y,\ \ \ \text{for all }t \in \mathbb{C}.
\end{equation}

Fix $q\in M$. Note that $X_q$ and $Y_q$ form a basis of $T_q^{1,0}M$. We may write
\[
(Y_t)_q=A_q(t) X_q + B_q(t) Y_q,
\]
where $A_q(t)=\alpha_q ((Y_t)_q)$, $B_q(t)=\beta_q ((Y_t)_q)$. Note that $X_t,Y_t$ are linearly independent everywhere while $X_t=(\Phi_t^X)^*X=X$. Then 
\begin{equation} \label{2by2_matrix}
\begin{pmatrix}
    1& A_q(t) \\
    0& B_q(t)
\end{pmatrix}    
\end{equation}
is invertible and $B_q(t)$ is nowhere vanishing. However, $B_q(t)$ is a polynomial of $t$. Indeed, it follows from \eqref{taylor} that 
\begin{equation*}
    B_q(t)=\beta_q ((Y_t)_q)=\beta_q (Y_q) +t\beta_q ([X,Y]_q) +\cdots+ \frac{t^{N-1}}{(N-1)!}\beta_q ((\ad_X^{N-1}Y)_q).
\end{equation*}
Therefore $B_q(t)$ has to be constant. Hence we get
\begin{equation*}
    0=\left.\frac{d}{dt}\right|_{t=0}B_q(t)=\beta_q ([X,Y]_q).
\end{equation*}
Hence $\beta ([X,Y])=0$ as the point $q$ is chosen arbitrarily.

The same argument using the complete flow of $Y$ gives
\(
\alpha ([X,Y])=0.
\)
Since $\alpha$ and $\beta$ form a global coframe, we have $[X,Y]=0$. It follows from \eqref{eq:torsion-coefficients} that the Chern torsion of $(M, g)$ vanishes everywhere on $M$. We conclude that $(M,g)$ is holomorphically isometric to $\mathbb{C}^2$.

\end{proof}

\begin{remark}\label{rem:nilpotent}
Note that the proof of Theorem \ref{thm:surface-rigidity} relies on the fact that if the $2 \times 2$ matrix in \eqref{2by2_matrix} is invertible, then $B_q(t)$ is nowhere zero. The same argument in dimension $\geq 3$ breaks down as we could have an invertible matrix such as
$\begin{pmatrix}
  1 & 1 & 1  \\
  0 & t & 1  \\
  0&  t^2-1 & t
\end{pmatrix}$.    
On the other hand, in \cite[Proposition 5.4]{MWY2026}, we construct a complete Chern-flat metric on $\mathbb{C}^3$ which satisfies Property (H); see also Remark \ref{rem:no-isometric-conclusion}. This metric is not Euclidean as long as $\frac{\partial \xi}{\partial z_1}-\frac{\partial \eta}{\partial z_2}$ in \eqref{eta_xi} is non-zero. Moreover, any simply connected nilpotent complex Lie group equipped with a left-invariant metric admits global holomorphic coordinates of polynomial growth; see \cite[Theorem 4.13]{MWY2026} for more details. Note that the determinant of the above $3 \times 3$ matrix is the constant $1$. We will exploit this fact in the following proof of Theorem \ref{thm:high-dimension}.
\end{remark}

\begin{proof}[Proof of Theorem \ref{thm:high-dimension}: Part II]

The conclusion that $g$ is balanced is partly motivated by \cite[Proposition 5.5]{MWY2026}, where we observe that any left-invariant metric on a complex nilpotent Lie group is balanced. However, the method is inspired by that in the proof of Theorem~\ref{thm:surface-rigidity}, where a polynomial expansion as in \eqref{taylor} holds on a nilpotent Lie algebra.

We consider the unitary parallel frame $\{E_1, \ldots, E_n\}$ on $(M, g)$ provided by Lemma \ref{thm:boothby-frame}.
According to Proposition \ref{prop:nilpotent-brackets}, the Lie algebra generated by these $\{E_1, \ldots, E_n\}$ is finite-dimensional and nilpotent. We have $\operatorname{ad}_{E_i} E_j=[E_i, E_j]$ well-defined and there exists some integer $N$ such that $ad_{E_i}^N E_j=0$ for any $1 \leq i, j \leq n$.

Fix $1 \leq i \leq n$ and let $\Phi_t$ denote the complete holomorphic flow of $E_i$. As $\operatorname{ad}_{E_i}(E_j)=[E_i, E_j]=-2T_{ij}^kE_k$ by \eqref{eq:torsion-coefficients}, we are led to consider \eqref{taylor} again.
\begin{equation}\label{taylor_high}
    (\Phi_t)^{\ast}(E_j)=E_j+t[E_i, E_j]+\cdots+\frac{t^{N-1}}{(N-1)!}\operatorname{ad}_{E_i}^{N-1}(E_j),\ \ \ 1 \leq j \leq n.
\end{equation}
Fix any point $p \in M$. Following \eqref{2by2_matrix}, we consider $P(t)$ which is the coefficient matrix when we express these $(\Phi_t)^{\ast}(E_j)$ with respect to the global frame $\{E_j\}_{j=1}^n$, both evaluated at $p$.
\begin{equation}\label{def_P(t)}
\Bigl((\Phi_t)^{\ast}(E_1), \ldots ,(\Phi_t)^{\ast}(E_n) \Bigr)=\Bigl(E_1, \ldots , E_n \Bigr) P(t).
\end{equation}
It follows from \eqref{taylor_high} that each entry of $P(t)$ is a polynomial of $t$ and $P(0)$ is the identity matrix. However, $P(t)$ is invertible for each $t \in \mathbb{C}$. We see that $\det (P(t)) \equiv 1$. It follows that
\[
0=\frac{d}{dt}\Big|_{t=0} \det (P(t))=\operatorname{tr}P'(0)=\sum_{j=1}^n \varphi^j([E_i, E_j])_{p}=-2\sum_{j=1}^n T_{ij}^j(p).
\]
Therefore, $g$ is a balanced metric.

\end{proof}

\section{Examples}\label{Sec5}

In this section, we provide the detailed construction of Example \ref{new_example_exp}. 

Let $G=\operatorname{SL}(2,\mathbb{C})$ and set 
\begin{equation} \label{d(s)_def}
 h=\begin{pmatrix}
    1& 0 \\ 0& -1
\end{pmatrix},\ \ \ \ d(s)=e^{2\pi i sh}=\begin{pmatrix}
    e^{2\pi i s}& 0 \\ 0& e^{-2\pi i s}
\end{pmatrix},\ \ \ \ D= \{d(s)\ |\ s \in \mathbb{C}\}.   
\end{equation}
Note that $d(s)=d(s')$ if and only if $s-s' \in \mathbb{Z}$, and $D$ is a one-parameter subgroup of $G$ which is isomorphic to $\mathbb{C}^{\ast}$. We note that the affine quadric $Q$ defined in \eqref{def_Q} is biholomorphic to $D \backslash G$. Namely, we pick $g=\begin{pmatrix}a&b\\c&d\end{pmatrix} \in G$,
and consider the conjugation action of $g$ on $h$
\begin{equation}\label{conjuga_act}
g^{-1}hg=\begin{pmatrix}ad+bc & 2bd\\ -2ac & -(ad+bc) \end{pmatrix}.  
\end{equation}
Note that $g^{-1}hg$ is constant on any fixed right coset $Dg \in D \backslash G$. We choose the particular form of $g^{-1}hg$ as it is compatible with the right action on coset $Dg$.  From now on, we express the affine quadric as
\begin{equation}\label{def_Q_second}
    Q=\{ q=\begin{pmatrix}x & y\\ z & -x \end{pmatrix} \in  \operatorname{GL}(2,\mathbb{C}) \ |\ \det(q)=-1\}.
\end{equation} We may define an explicit biholomorphism $\Phi$ between $D \backslash G$ and $Q$.
\begin{equation}\label{Q_complex_homo}
\Phi: D \backslash G  \rightarrow Q,\ \ \  \Phi\Bigl(D \begin{pmatrix}
    a & b \\  c & d
\end{pmatrix}\Bigr) = \begin{pmatrix}
    ad+bc & 2bd \\  -2ac & -(ad+bc)
\end{pmatrix}.
\end{equation} As $-(ad+bc)^2-(2bd)(-2ac)=-(ad-bc)^2=-1$, the right-hand side of \eqref{Q_complex_homo} indeed lies in $Q$. 

Our first goal in this section is to construct a complete Chern-flat metric on $Q \times \mathbb{C}$. It might be interesting in view of the remark following Question \ref{ques_intro}: it is unclear whether $Q$ admits a complete Chern-flat metric.

\subsection{A complete Chern-flat metric on $Q \times \mathbb{C}$}

We first choose a left-invariant frame on $G$. Let
\begin{equation}\label{S_base_def}
S_1=\frac{1}{\sqrt2}
\begin{pmatrix}1&0\\0&-1\end{pmatrix},
\qquad
S_2=\begin{pmatrix}0&1\\0&0\end{pmatrix},
\qquad
S_3=\begin{pmatrix}0&0\\1&0\end{pmatrix}    
\end{equation}
be a standard orthonormal basis of $\mathfrak{sl}(2,\mathbb C)$ equipped with the Hermitian inner product $\langle X,Y\rangle=\operatorname{tr}(X Y^{\ast})$ where $Y^{\ast}=\overline{Y^T}$. Then $S_1, S_2, S_3$ determine a left-invariant holomorphic frame on $G$ by $X_j(g)=\left.\frac{d}{dt}\right|_{t=0}g e^{tS_j}$ for $1 \leq j \leq 3$.
For $g=\begin{pmatrix}a&b\\c&d\end{pmatrix} \in G$ where $ad-bc=1$, we obtain
\begin{equation}\label{frame_on_G}
X_1=\frac{1}{\sqrt2}
\left(
 a\frac{\partial}{\partial a}
-b\frac{\partial}{\partial b}
+c\frac{\partial}{\partial c}
-d\frac{\partial}{\partial d}
\right),\ \ 
X_2=a\frac{\partial}{\partial b}
+c\frac{\partial}{\partial d},\ \ 
X_3=b\frac{\partial}{\partial a}
+d\frac{\partial}{\partial c}.
\end{equation}
Next, we compute the vector fields induced by $X_1,X_2,X_3$ on $Q$. Consider the map $\pi:G \to Q$ obtained by composing the natural projection $G \to D \backslash G$ with $\Phi$. Then $\pi(g)=g^{-1}hg$ as in \eqref{conjuga_act}. Now we consider $Z_j=\pi_{\ast}(X_j)$ defined by
\begin{equation}\label{Z_formula_1}
Z_j(\pi(g))=\pi_{*g}\bigl(X_j(g)\bigr).    
\end{equation}
As the integral curve of $X_j$ is $g(t)=g e^{tS_j}$, we have
\(
\pi(g(t)) =e^{-tS_j}g^{-1}hg e^{tS_j} =e^{-tS_j}\pi(g)e^{tS_j}.
\)
After differentiating at $t=0$, we obtain
\begin{equation}\label{tangent_vec_f}
\left.\frac{d}{dt}\right|_{t=0}\pi(g(t))
=[\pi(g),S_j].    
\end{equation}
Therefore, for every $q \in Q$ defined in \eqref{def_Q_second}, $Z_j(q)=[q,S_j]$ for $1 \leq j \leq 3$. We solve
\[
[q,S_1] =\begin{pmatrix}
0&-\sqrt2\,y\\
\sqrt2\,z&0
\end{pmatrix}, \qquad 
[q,S_2]=\begin{pmatrix}
-z&2x\\
0&z
\end{pmatrix}, \qquad
[q,S_3]
=\begin{pmatrix}
y&0\\
-2x&-y
\end{pmatrix}.
\]
Consequently,
\begin{align}\label{Z_formula}
Z_1=-\sqrt2\,y\frac{\partial}{\partial y}
+\sqrt2\,z\frac{\partial}{\partial z},\ \ \
Z_2=-z\frac{\partial}{\partial x}
+2x\frac{\partial}{\partial y},\ \ \
Z_3=y\frac{\partial}{\partial x}
-2x\frac{\partial}{\partial z}.
\end{align}
Observe that $Z_1$, $Z_2$, $Z_3$ are linearly dependent at every point, as they satisfy
\begin{equation}\label{linear-relation}
\sqrt2\,xZ_1+yZ_2+zZ_3=0.
\end{equation}
Let $u$ be the standard coordinate on $\mathbb C$, with the corresponding coordinate vector field $\partial/\partial u$. Then $T_{(q,u)}(Q\times\mathbb C)=T_qQ\oplus\mathbb C$.
Set $E_j=Z_j+f_j\frac{\partial}{\partial u}$ for $1 \leq j \leq 3$, where each $f_j$ is a holomorphic function on $Q$. We look for functions $f_1$, $f_2$, and $f_3$ such that $E_1$, $E_2$, and $E_3$ are linearly independent everywhere. By \eqref{linear-relation},
\(
\sqrt2\,xE_1+yE_2+zE_3=\left(\sqrt2\,xf_1+yf_2+zf_3\right)\frac{\partial}{\partial u}.
\)
Therefore, $E_1$, $E_2$, $E_3$ are linearly independent everywhere if and only if
\(
\sqrt2\,xf_1+yf_2+zf_3\ne 0.
\)
As $x^2+yz=1$, it is natural to take $f_1=\frac{x}{\sqrt2}, f_2=\frac{z}{2}$, and $f_3=\frac{y}{2}$. Then we obtain the following global holomorphic frame on $Q \times\mathbb C$:
\begin{equation}\label{frame_E_i}
E_1=-\sqrt2\,y\frac{\partial}{\partial y}
+\sqrt2\,z\frac{\partial}{\partial z}
+\frac{x}{\sqrt2}\frac{\partial}{\partial u},\
E_2=-z\frac{\partial}{\partial x}
+2x\frac{\partial}{\partial y}
+\frac{z}{2}\frac{\partial}{\partial u},\ 
E_3=y\frac{\partial}{\partial x}
-2x\frac{\partial}{\partial z}
+\frac{y}{2}\frac{\partial}{\partial u}.
\end{equation}
Let $\varphi^1, \varphi^2, \varphi^3$ be the corresponding dual holomorphic $(1, 0)$-coframe. A direct calculation shows
\begin{align}\label{frame_varphi_i}
    \varphi^1=-\frac{z}{2\sqrt{2}}dy+\frac{y}{2\sqrt{2}}dz+\sqrt{2}x\,du,\ \ \varphi^2=-\frac{y}{2}dx+\frac{x}{2}dy+ydu,\ \ \ \varphi^3=\frac{z}{2}dx-\frac{x}{2}dz+zdu.
\end{align}
Now we define the Hermitian metric
\begin{equation}\label{varphi_metric}
    \omega=\sqrt{-1} (\varphi^1 \wedge \overline{\varphi^1}+ \varphi^2 \wedge \overline{\varphi^2}+ \varphi^3 \wedge \overline{\varphi^3}).
\end{equation}

\begin{lemma}\label{QC_example}
    $(Q \times \mathbb{C}, \omega)$ is a simply connected complete Chern-flat manifold. In particular, $Q \times \mathbb{C}$ admits no complex Lie group structure.
\end{lemma}

\begin{proof}[Proof of Lemma \ref{QC_example}]
It follows from the exact sequence $D \to G \to Q$ that $\pi_1(Q)=0$ and $\pi_2(Q)\simeq \mathbb{Z}$. Hence $\pi_1(Q\times \mathbb{C})=0$ and $\pi_2(Q \times \mathbb{C})\simeq \mathbb{Z}$. Since the second homotopy group of a connected Lie group is trivial, $Q \times \mathbb{C}$ admits no complex Lie group structure.

    It suffices to verify the completeness of $\omega$. 
Consider a $C^1$ curve
\begin{equation}\label{gamma_curve}
    \gamma:[0, t_1]\longrightarrow Q\times\mathbb C,
\qquad \gamma(t)=(q(t),u(t)), \qquad  \gamma(0)=(h,0),
\end{equation}
where
\(
q(t)=\begin{pmatrix}
x(t)&y(t)\\
z(t)&-x(t)
\end{pmatrix}.
\)
We assume that the $(1,0)$-part of $\gamma'$ can be expressed using the frame $E_1$, $E_2$, $E_3$.
\begin{equation}\label{frame-equation}
(\gamma')^{1,0}(t)=x'(t) \frac{\partial}{\partial x}+y'(t) \frac{\partial}{\partial y}+z'(t) \frac{\partial}{\partial z}+u'(t) \frac{\partial}{\partial u} = \xi_1(t)E_1+\xi_2(t)E_2+\xi_3(t)E_3.
\end{equation}
Let $a(t)= \sqrt{|\xi_1(t)|^2 + |\xi_2(t)|^2+|\xi_3(t)|^2}$. Then the length of $\gamma$ is
\(
L\left(\gamma(\tau)\big|_{[0, t]}\right)=\sqrt2\int_0^{t} a(\tau) d\tau.
\)
Substituting the previously obtained expressions for $E_1$, $E_2$, $E_3$ into \eqref{frame-equation}, we obtain
\begin{equation}\label{xyz_prime(t)_f}
\begin{split}
x'(t)&=-z\xi_2+y\xi_3, \ \ \ \ y'(t)=-\sqrt2 y\xi_1+2x\xi_2,\\
z'(t)&=\sqrt2 z\xi_1-2x\xi_3,\ \ \ \ u'(t)=\frac{x}{\sqrt2}\xi_1+\frac{z}{2}\xi_2+\frac{y}{2}\xi_3. 
\end{split}
\end{equation}
From \eqref{S_base_def}, we introduce
\begin{equation}\label{A(t)_f}
A(t)
=\begin{pmatrix}
\dfrac{1}{\sqrt2}\xi_1(t)& \xi_2(t) \\
\xi_3(t)&-\dfrac{1}{\sqrt2}\xi_1(t)
\end{pmatrix} =\xi_1(t)S_1+\xi_2(t)S_2+\xi_3(t)S_3.    
\end{equation}
Motivated by \eqref{tangent_vec_f}, we may rewrite \eqref{xyz_prime(t)_f} using \eqref{A(t)_f}.
\begin{equation}\label{q_u_prime_t_f}
q'(t)=\begin{pmatrix}
x'(t)&y'(t)\\
z'(t)&-x'(t)
\end{pmatrix}=[q(t),A(t)],\ \ \ \ \ u'(t)=\frac12\operatorname{tr}\bigl(A(t)\,q(t)\bigr).   
\end{equation}
We analyze the system \eqref{q_u_prime_t_f} by the matrix norm. Recall that the Frobenius norm of a matrix $C=(C_{ij})_{n\times n}$ is defined as
\(
\|C\|=\sqrt{\sum_{i,j=1}^n|C_{ij}|^2}.
\)
We have
\(
\|A(t)\|^2=|\xi_1(t)|^2+|\xi_2(t)|^2+|\xi_3(t)|^2=a(t)^2.
\)
Observe that
\[
\|q'(t)\|=\|qA-Aq\| \leq\|qA\|+\|Aq\| \leq2\|q\| \|A\| =2\|q(t)\| a(t).
\]
Therefore, for any $t>0$, we have
\begin{equation} \label{Gronwall_1}
    \|q(t)\|= \bigl\| q(0)+\int_0^t q'(\tau)d\tau \bigr\| \leq \|q(0)\|+\int_0^t \|q'(\tau)\|d\tau \leq \|q(0)\|+\int_0^t 2 a(\tau) \,\|q(\tau)\|d\tau.
\end{equation}
By the Gronwall inequality, we obtain
\begin{equation}\label{q(t)bound}
\|q(t)\| \leq\|q(0)\|\exp \left(2\int_0^t a(\tau) d\tau\right) \leq\sqrt2 e^{\sqrt2 L(\gamma|_{[0, t]})}.    
\end{equation}
Then we estimate the growth of $u(t)$ using \eqref{q(t)bound} and the Cauchy-Schwarz inequality.
\[
|u'(t)|
=\frac{1}{2}|\operatorname{tr}(Aq)| \leq\frac{1}{2}\|q(t)\| \|A(t)\| \leq \frac{1}{\sqrt2} a(t)
\exp \left(2\int_0^t a(s) ds\right).
\]
Note that $u(0)=0$. After integrating both sides, we conclude that
\begin{equation}\label{u(t)bound}
|u(t)|\leq\frac{1}{2\sqrt2}\,e^{\sqrt{2}L(\gamma|_{[0, t]})}.
\end{equation}
It follows from \eqref{q(t)bound} and \eqref{u(t)bound} that $\omega$ is complete.
\end{proof}

\begin{lemma}\label{QC_torsion}
    The norm of the torsion of $(Q \times \mathbb{C}, \omega)$ has precisely exponential growth.
\end{lemma}

\begin{proof}[Proof of Lemma \ref{QC_torsion}]
The torsion tensor of $\omega$ can be computed using either \eqref{eq:first-structure} or \eqref{eq:torsion-coefficients}.
\begin{equation}\label{torsion_comp_f}
\begin{split}
T(E_1,E_2)&=-xzE_1-\frac{2+yz}{\sqrt2}E_2-\frac{z^2}{\sqrt2}E_3,\ \ T(E_1,E_3)=xyE_1+\frac{y^2}{\sqrt2}E_2+\frac{2+yz}{\sqrt2}E_3,\\
T(E_2,E_3)&=-\sqrt2(1+x^2)E_1-xyE_2-xzE_3.
\end{split}
\end{equation}
Let $\alpha=2|x|^2+|y|^2+|z|^2$. Since $x^2+yz=1$, we obtain
\begin{equation}\label{torsion_norm_f}
|T|^2=10+\frac{1}{2}\bigl(2|x|^2+|y|^2+|z|^2\bigr)^2=10+\frac{1}{2}\alpha^2.    
\end{equation}

Let $B((h, 0), R)$ denote the geodesic ball centered at the point $(h, 0)$ with radius $R$. Our goal is to estimate 
$\sup_{B((h, 0), R)}|T|^2$. To that end, we fix a point $(q, u)\in Q\times\mathbb C$, and choose a curve $\gamma$ with $t \in [0, t_1]$ as in \eqref{gamma_curve} such that $\gamma(t_1)=(q, u)$. Differentiating $\alpha=2|x|^2+|y|^2+|z|^2$ gives
\[
\begin{aligned}
\alpha'(t)=2\operatorname{Re}\Bigl(\sqrt2\bigl(|z|^2-|y|^2\bigr)\xi_1
 +2(x\overline y-\overline x z)\xi_2+2(\overline x y-x\overline z)\xi_3\Bigr).
\end{aligned}
\]
By the Cauchy--Schwarz inequality, we get
\[
|\alpha'(t)|\leq2\sqrt2\, a(t)\, \sqrt{\bigl(|z|^2-|y|^2\bigr)^2+4|x\overline y-\overline x z|^2},
\]where $a(t)$ is defined after \eqref{frame-equation}.
Since $x^2+yz=1$ holds, we have $|yz|^2=|1-x^2|^2$. It follows that
\(
\bigl(|z|^2-|y|^2\bigr)^2+4|x\overline y-\overline x z|^2=\alpha^2-4.
\)
Hence
\(
|\alpha'(t)| \leq 2\sqrt{\alpha^2-4}\cdot|\gamma'|.
\)
Let \(\beta=\frac12\operatorname{arcosh}\frac{\alpha}{2}\). Then
\[
|\beta'|=\frac{|\alpha'|}{2\sqrt{\alpha^2-4}}\leq|\gamma'|.
\]
As $\alpha(h,0)=2$ and $\beta(h,0)=0$, we obtain $\beta(p)\leq L(\gamma)$. After we take the infimum over the lengths of any such curves, we arrive at
\[
\beta(p)\leq d\bigl((h,0),p\bigr), \ i.e.,\  \alpha(p)\leq2\cosh\bigl(2d((h,0),p)\bigr).
\]
Therefore, we have proved that
\begin{equation}\label{torsion_exp_f1}
\sup_{p\in B((h,0),R)}|T(p)|^2 \leq 10+2\cosh^2(2R).    
\end{equation}
We proceed to show that the equality holds in \eqref{torsion_exp_f1}. Consider
\[
q(t)=
\begin{pmatrix}
\cosh t&\sinh t\\
-\sinh t&-\cosh t
\end{pmatrix}.
\]
Clearly, $q(t)\in Q$ and $q(0)=h$. Set $\gamma(t)=(q(t),0)$. For such a special curve, we may solve \eqref{xyz_prime(t)_f} and obtain $\xi_1=0$ and $\xi_2=\xi_3=\frac{1}{2}$. Hence $|\gamma'(t)| \equiv 1$ and
$L\left(\gamma|_{[0,t]}\right)=t$. Moreover, $\alpha(t)=2\cosh(2t)$. It follows that $\beta(t)=t\leq d\bigl((h,0),\gamma(t)\bigr)$. Therefore, $\gamma(t)$ is length-minimizing in the sense that
\[
L \left(\gamma|_{[0,t]}\right)=t
=d\bigl((h,0),\gamma(t)\bigr).
\]
It follows that $\bigl|T(\gamma(R))\bigr|^2
=10+2\cosh^2(2R)$. 

\end{proof}

Recall $G=\operatorname{SL}(2,\mathbb{C})$ and $D$ is defined in \eqref{d(s)_def}. Note that $\pi_2(G)=0$ while $\pi_2(Q \times \mathbb{C}) \simeq \mathbb{Z}$. Although $Q \times \mathbb{C}$ and $G$ are not diffeomorphic to each other, they are related by a continuous one-parameter variation. To see it, for a given $\lambda\in \mathbb{C}$, let
\begin{equation}\label{H_lambda_def}
H_{\lambda}=\left\{ (d(s),-s\lambda) \;\middle| \; s \in \mathbb{C}, d(s)\in D \right\}    
\end{equation}
which is a closed subgroup of $G \times \mathbb{C}$. We consider a left action of $H_{\lambda}$ on $G \times \mathbb{C}$ 
\begin{equation}\label{H_lambda_act}
(d(s), -s\lambda) \cdot (g,u)=(d(s)g, u-s\lambda).
\end{equation}
When $\lambda=0$, $H_0$ is isomorphic to $\mathbb{C}^{\ast}$. By \eqref{Q_complex_homo}, we know that $H_0\backslash (G\times \mathbb{C})$ is biholomorphic to $Q \times \mathbb{C}$. However, $H_{\lambda}$ is isomorphic to the additive group $\mathbb{C}$ if $\lambda \neq 0$. The corresponding $M_{\lambda} \doteq H_{\lambda}\backslash (G\times \mathbb{C})$ is biholomorphic to $G$. Indeed, one may check that
\begin{equation}\label{F_lambda_def}
F_{\lambda}: M_{\lambda} \to G, \qquad F_{\lambda}([g,u])=d(\frac{u}{\lambda})g \end{equation}
defines the desired biholomorphism for $\lambda \neq 0$. When $\lambda$ tends to $0$, $H_{\lambda}$ does not converge to $H_0$ in the sense of convergence of closed subgroups of $G \times \mathbb{C}$. For example, pick $\lambda=\frac{1}{n}$ and $s=n$ with a positive integer $n$, and then $(I_2, -1) \in H_{\lambda}$ where $I_2$ is the identity matrix in $G$. Obviously $(I_2, -1) \notin H_0$.

For any $\lambda \neq 0$, there exists a natural global frame on $M_{\lambda}$. It coincides with a global left-invariant frame on $G$ through the biholomorphism $F_{\lambda}$. Namely, we consider 
\[
\begin{pmatrix}
    \widehat{a} & \widehat{b} \\
    \widehat{c} & \widehat{d}
\end{pmatrix}=d(\frac{u}{\lambda})\begin{pmatrix}
    a & b \\
    c & d
\end{pmatrix}=\begin{pmatrix}
    e^{2\pi i \frac{u}{\lambda}}a & e^{2\pi i \frac{u}{\lambda}}b \\
    e^{-2\pi i \frac{u}{\lambda}}c & e^{-2\pi i \frac{u}{\lambda}}d
\end{pmatrix}.
\]
We see that $\Bigl(\begin{pmatrix}
    a & b \\
    c & d
\end{pmatrix}, u\Bigr) \rightarrow \widehat{a}=e^{2\pi i\frac{u}{\lambda}}a$, viewed as a holomorphic function on $G \times \mathbb{C}$, is $H_{\lambda}$-invariant. Similarly, we see that $\widehat{a}, \widehat{b}, \widehat{c}$, and $\widehat{d}$ are global holomorphic functions on $M_{\lambda}$. Recall that the left invariant frame $\{X_j\}_{j=1}^3$ is defined in \eqref{frame_on_G}. Motivated by the formula of $\{X_j\}_{j=1}^3$, we write down a global holomorphic frame on $M_{\lambda}$. 
\begin{equation}\label{frame_on_quotient}
\widehat{X}_1=\frac{1}{\sqrt2}
\left(
 \widehat{a}\frac{\partial}{\partial \widehat{a}}
-\widehat{b}\frac{\partial}{\partial \widehat{b}}
+\widehat{c}\frac{\partial}{\partial \widehat{c}}
-\widehat{d}\frac{\partial}{\partial \widehat{d}}
\right),\ \ 
\widehat{X}_2=\widehat{a}\frac{\partial}{\partial \widehat{b}}
+\widehat{c}\frac{\partial}{\partial \widehat{d}},\ \ 
\widehat{X}_3=\widehat{b}\frac{\partial}{\partial \widehat{a}}
+\widehat{d}\frac{\partial}{\partial \widehat{c}}.
\end{equation}
We may check that $dF_{\lambda}(\widehat{X}_j)=X_j$ for any $1 \leq j \leq 3$. Therefore, for each $\lambda\neq 0$, let $\{\widehat{\varphi}^{i}\}_{i=1}^3$ denote the dual coframe with respect to \eqref{frame_on_quotient}. Then
\begin{equation}\label{widehat_varphi_metric}
    \omega_{\lambda}=\sqrt{-1} (\widehat{\varphi}^1 \wedge \overline{\widehat{\varphi}^1}+ \widehat{\varphi}^2 \wedge \overline{\widehat{\varphi}^2}+ \widehat{\varphi}^3 \wedge \overline{\widehat{\varphi}^3})
\end{equation}
defines a complete Chern-flat Hermitian metric $\omega_{\lambda}$ on $M_{\lambda}$, and $F_{\lambda}$ is a holomorphic isometry between $\omega_{\lambda}$ and the standard left-invariant Hermitian metric $\omega_0$ on $G$. In other words, we observe the following result.
\begin{lemma}\label{GH_converge}
Let $p_0 \doteq [I_2, 0] \in M_{\lambda}$ and $\operatorname{1}_G$ denote the identity element in $G$. Then as $\lambda \rightarrow 0$, $(M_{\lambda}, \omega_{\lambda}, p_0)$ converges to $(G, \omega_0, \operatorname{1}_G)$ in the pointed Gromov--Hausdorff topology.    
\end{lemma}

It is possible to analyze the geometry of $M_{\lambda}$ in more detail. We set up some notations on several related maps. Let $q: G \times \mathbb{C} \rightarrow M_{\lambda}$ be the quotient map $q(g, u)=[g, u]$ and $\operatorname{pr}_1: G \times \mathbb{C} \rightarrow G$ be the projection to the first component. Recall that $\pi: G \rightarrow Q$ is defined in \eqref{Z_formula_1}. Then the quotient of $\operatorname{pr}_1$ by the left action of $H_{\lambda}$ defined in \eqref{H_lambda_act} results in $\rho: M_{\lambda} \rightarrow Q$ which is a holomorphic $\mathbb{C}^{\ast}$ bundle map. 
\begin{equation}\label{diagram_maps}
\begin{tikzcd}[column sep=40pt, row sep=30pt]
G \times \mathbb{C}
    \arrow[r, "q"]
    \arrow[d, "\operatorname{pr}_1"']
  & M_{\lambda}
    \arrow[d, "\rho"] \\
G
    \arrow[r, "\pi"]
  & Q
\end{tikzcd}
\end{equation}
Now we write down local trivializations of $\rho$. Let 
\[
U_{+}=\{x \neq -1\ |\ (x, y, z) \in Q\},\ \ \ U_{-}=\{x \neq 1\ |\ (x, y, z) \in Q\}.
\]
Then $Q = U_{+} \cup U_{-}$. Recall that $x=ad+bc, y=2bd, z=-2ac$, and $ad-bc=1$ by \eqref{Q_complex_homo}.
On $U_{+}$, we have $ad=\frac{1+x}{2} \neq 0$, and then we can normalize $a=1$ and solve
\begin{equation}\label{local_trivial_1}
\rho^{-1}(U_{+})=\Bigl[\begin{pmatrix}
    1  &  \frac{y}{1+x} \\ 
    -\frac{z}{2}  &  \frac{1+x}{2}
\end{pmatrix}, u_{+} \doteq u+\frac{\lambda}{2\pi i} \log a \Bigr] \subset M_{\lambda}.    
\end{equation}
Here we take a fixed principal branch of $\log a$ and $u_1$ takes a value in $\mathbb{C} / \lambda\mathbb{Z}$. The local expressions of the frame \eqref{frame_on_quotient} are
\begin{equation}\label{frame_on_quotient_1}
\begin{split}
\widehat{X}_1&=-\sqrt2\,y\frac{\partial}{\partial y}+\sqrt2\,z\frac{\partial}{\partial z}
+ \frac{\lambda}{2\pi i} \frac{1}{\sqrt{2}} \frac{\partial}{\partial u_{+}},\ \ \
\widehat{X}_2=-z\frac{\partial}{\partial x}
+2x\frac{\partial}{\partial y},\\
\widehat{X}_3&=y\frac{\partial}{\partial x}-2x\frac{\partial}{\partial z}
+\frac{\lambda}{2\pi i} \frac{y}{1+x} \frac{\partial}{\partial u_{+}}.     
\end{split}
\end{equation}
Similarly, on $U_{-}$ we have $bc=\frac{x-1}{2} \neq 0$, and then we can normalize $b=1$ and express
\begin{equation}\label{local_trivial_2}
\rho^{-1}(U_{-})=\Bigl[\begin{pmatrix}
    \frac{z}{1-x}  &  1 \\ 
    \frac{x-1}{2}  &  \frac{y}{2}
\end{pmatrix}, u_{-} \doteq u+\frac{\lambda}{2\pi i} \log b \Bigr] \subset M_{\lambda}.
\end{equation}
The corresponding local expressions of the frame \eqref{frame_on_quotient} are
\begin{equation}\label{frame_on_quotient_2}
\begin{split}
\widehat{X}_1&=-\sqrt2\,y\frac{\partial}{\partial y}+\sqrt2\,z\frac{\partial}{\partial z}
-\frac{\lambda}{2\pi i} \frac{1}{\sqrt{2}} \frac{\partial}{\partial u_{-}},\ \ \
\widehat{X}_2=-z\frac{\partial}{\partial x}
+2x\frac{\partial}{\partial y}+\frac{\lambda}{2\pi i} \frac{z}{1-x} \frac{\partial}{\partial u_{-}},\\
\widehat{X}_3&=y\frac{\partial}{\partial x}-2x\frac{\partial}{\partial z}.    
\end{split}
\end{equation}
Note that the transition function along the fiber is given by $u_{-}=u_{+}+\frac{\lambda}{2\pi i} \log \frac{y}{1+x}$ on $\rho^{-1}(U_+ \cap U_{-})$. We may also check directly that \eqref{frame_on_quotient_1} and \eqref{frame_on_quotient_2} coincide on the overlap. Moreover, $\frac{\partial}{\partial u_1}$ and $\frac{\partial}{\partial u_2}$ glue to a global vector field $\frac{\partial}{\partial u}$. Next we introduce two local holomorphic $1$-forms
\begin{equation}\label{eta_two_locals}
    \widehat{\eta}_1=du_{+}+\frac{\lambda}{2\pi i} \frac{z}{2} d(\frac{y}{1+x}),\ \ \text{on }\rho^{-1}(U_{+}),\ \ \
    \widehat{\eta}_2=du_{-}+\frac{\lambda}{2\pi i} \frac{y}{2} d(\frac{z}{1-x}),\ \ \text{on }\rho^{-1}(U_{-}).
\end{equation} By a direct calculation, $\widehat{\eta}_1$ and $\widehat{\eta}_2$ are equal on $\rho^{-1}(U_{+} \cap U_{-})$. Hence they define a global holomorphic $1$-form $\widehat{\eta}$ on $M_{\lambda}$. Using the quotient map $q$ defined in \eqref{diagram_maps}, one may find a nicer form of the pullback of $\widehat{\eta}$ on $G \times \mathbb{C}$. Indeed, $q^{\ast}(\widehat{\eta})=du+\frac{\lambda}{2\pi i} \bigl(d (\operatorname{d}a)-c(\operatorname{d}b)\bigr)$, where we use $\operatorname{d}$ to denote the differential of $a, b, c,$ and $d$. Let $g=\begin{pmatrix}a&b\\c&d\end{pmatrix} \in G$. As $ad-bc=1$, we rewrite $q^{\ast}(\widehat{\eta})$ as
\begin{equation}\label{eta_pull_back}
    q^{\ast}(\widehat{\eta})=du+\frac{\lambda}{4\pi i} \operatorname{tr}\bigl(h (\operatorname{d}g) g^{-1}\bigr),
\end{equation} where $h$ is defined in \eqref{d(s)_def}.

Recall that the $\{\widehat{\varphi}^i\}_{i=1}^3$ is the dual coframe of $\{\widehat{X}^i\}_{i=1}^3$. We may solve
\begin{equation}\label{eta_rep}
    \widehat{\eta}=\frac{\lambda}{2\pi i} \Bigl( \frac{x}{\sqrt{2}} \widehat{\varphi}^1+\frac{z}{2} \widehat{\varphi}^2+\frac{y}{2} \widehat{\varphi}^3 \Bigr).
\end{equation}
Note that $\widehat{\eta}$ defines a holomorphic connection on $M_{\lambda}$ with a holomorphic horizontal distribution $\operatorname{ker}(\widehat{\eta})$. We define the corresponding horizontal lift for any $V \in T^{1, 0}Q$ as 
\begin{equation}\label{horizontal_lift_1}
(V)^{H}=(V)-\widehat{\eta}\bigl((V)\bigr)\frac{\partial}{\partial u},   \end{equation} where $(V) \in T^{1, 0} M_{\lambda}$ denotes the local coordinate lift of $V$. Recall that $\{Z_i\}_{i=1}^3$ is defined in \eqref{Z_formula}. We may rewrite $\{\widehat{X}_i\}_{i=1}^3$ in \eqref{frame_on_quotient} in terms of the horizontal lift.
\begin{equation}\label{frame_on_quotient_3}
\begin{split}
\widehat{X}_1=(Z_1)^{H}+ \frac{\lambda}{2\pi i} \frac{x}{\sqrt{2}} \frac{\partial}{\partial u},\ \ \
\widehat{X}_2=(Z_2)^{H}+ \frac{\lambda}{2\pi i} \frac{z}{2} \frac{\partial}{\partial u}, \ \ \widehat{X}_3=(Z_3)^{H}+ \frac{\lambda}{2\pi i} \frac{y}{2} \frac{\partial}{\partial u}.     
\end{split}
\end{equation}
As an application of the holomorphic connection defined by $\widehat{\eta}$, we see that $(M_{\lambda}, \omega_{\lambda})$ defined in \eqref{widehat_varphi_metric} and $(Q \times \mathbb{C}, \omega)$ in \eqref{varphi_metric} induce the same Hermitian metric on $Q$. 
\begin{lemma}\label{omega_Q_identify}
   $(M_{\lambda}, \omega_{\lambda})$ and $\widehat{\eta}$ induce a Hermitian metric $\omega_Q$ on $Q$. Moreover, $\omega_{Q}$ is independent of $\lambda \neq 0$. In fact, $\omega_{Q}$ coincides with the induced metric of $(Q \times \mathbb{C}, \omega)$ on the slice $Q \times \{0\}$. As $(Q \times \mathbb{C}, \omega)$ is complete, $(Q, \omega_Q)$ is also complete. Conversely, there is a natural way to reconstruct $(M_{\lambda}, \omega_{\lambda})$ based on $(Q \times \mathbb{C}, \omega)$ and $\widehat{\eta}$.
\end{lemma}

\begin{proof}[Proof of Lemma \ref{omega_Q_identify}]
    By restricting the frame \eqref{frame_varphi_i} to the slice $Q \times \{0\}$, we get
    \begin{align}\label{frame_mu_i}
    \mu^1=-\frac{z}{2\sqrt{2}}dy+\frac{y}{2\sqrt{2}}dz,\ \ \ \mu^2=-\frac{y}{2}dx+\frac{x}{2}dy,\ \ \ \mu^3=\frac{z}{2}dx-\frac{x}{2}dz.
\end{align}
    We study the pull-back $\rho^{\ast} (\mu^1)$. For example, we solve from \eqref{frame_on_quotient_3} 
    \[
    \rho^{\ast}(\mu^1)(\widehat{X}_1)=\mu^1(Z_1)=\Bigl(-\frac{z}{2\sqrt{2}}dy+\frac{y}{2\sqrt{2}}dz\Bigr)\Bigl(-\sqrt{2}(y\frac{\partial}{\partial y}-z\frac{\partial}{\partial z}) \Bigr)=yz.
    \]
    Comparing with \eqref{eta_rep}, we get
    \begin{equation}\label{pull_back_mu}
        \rho^{\ast}(\mu^1)=\widehat{\varphi}^1-\frac{2\pi i}{\lambda} \sqrt{2}x \widehat{\eta},\ \ \ \rho^{\ast}(\mu^2)=\widehat{\varphi}^2-\frac{2\pi i}{\lambda} y \widehat{\eta},\ \ \ \rho^{\ast}(\mu^3)=\widehat{\varphi}^3-\frac{2\pi i}{\lambda} z \widehat{\eta}.
    \end{equation}
    Pick any $Y, Z \in T^{1,0} Q$, we define $g_{Q}(Y, Z)=g_{\lambda}\bigl((Y)^H, (Z)^H\bigr)$. By \eqref{pull_back_mu}, we get
    \begin{equation}
        (Y)^H=\sum_{k=1}^3 \widehat{\varphi}^k\bigl((Y)^H\bigr) \widehat{X}^k=\sum_{k=1}^3 \rho^{\ast}(\mu^k)\bigl((Y)^H\bigr) \widehat{X}^k=\sum_{k=1}^3 \mu^k(Y) \widehat{X}^k.
    \end{equation}
    Consequently, we get
    \[
    \omega_Q=\sqrt{-1} \Bigl(\sum_{k=1}^3 \mu^k \wedge \overline{\mu^k}\Bigr)
    \] which coincides with the induced metric of $(Q \times \mathbb{C}, \omega)$ in \eqref{varphi_metric} along a slice $Q \times \{u\}$ for any $u \in \mathbb{C}$.

    Next we consider the reverse direction of recovering $(M_{\lambda}, \omega_{\lambda})$. Indeed, we may use $\omega_{Q}$ and $\widehat{\eta}$ to construct complete Hermitian metrics on the total space of $\rho: M_{\lambda} \rightarrow Q$. A natural choice will be $\rho^{\ast}\omega_Q+c_1 \sqrt{-1} \widehat{\eta} \wedge \overline{\widehat{\eta}}$ for a constant $c_1>0$. However, in order to recover the complete Chern-flat Hermitian metric $\omega_{\lambda}$ defined in \eqref{widehat_varphi_metric}, we use $(Q \times \mathbb{C}, \omega)$ and $\widehat{\eta}$ instead. First of all, we write $g_{Q \times \mathbb{C}}$ as the corresponding metric on $(Q \times \mathbb{C}, \omega)$. Note that $g_{Q \times \mathbb{C}}$ is invariant along the $\mathbb{C}$-direction and $g(\frac{\partial}{\partial u}, \frac{\partial}{\partial u})$ is independent of $u$. For any $\widehat{V} \in T_{[g, u]} M_{\lambda}$, we can view $d\rho(\widehat{V})+\frac{2\pi i}{\lambda}\widehat{\eta}(\widehat{V})\frac{\partial}{\partial u} \in T_{(\pi(g), 0)} (Q \times \mathbb{C})$. We define a Hermitian metric $\widetilde{g}$ on $M_{\lambda}$.
\begin{equation}
     \widetilde{g}(\widehat{V}, \widehat{W})=g_{Q \times \mathbb{C}}(d\rho(\widehat{V})+\frac{2\pi i}{\lambda}\widehat{\eta}(\widehat{V})\frac{\partial}{\partial u}, d\rho(\widehat{W})+\frac{2\pi i}{\lambda}\widehat{\eta}(\widehat{W})\frac{\partial}{\partial u}),\ \ \ \ \widehat{V}, \widehat{W} \in T_{[g, u]} M_{\lambda}.
\end{equation}
According to the apparent similarity between \eqref{frame_E_i} and \eqref{frame_on_quotient_3}, we get that $\{\widehat{X}_k\}_{k=1}^3$ is a global unitary frame with respect to $\widetilde{g}$. It follows that $\widetilde{g}$ and $g_{\lambda}$ coincide with each other.
     
\end{proof}

\begin{remark}\label{omega_Q_curvature}
    By a further calculation, one can check that $(Q, \omega_Q)$ is conformally K\"ahler. It has nonpositive (non-flat) Chern bisectional curvature and zero holomorphic sectional curvature. It remains unclear whether $Q$ can admit a complete Chern-flat metric.
\end{remark}

\subsection{Examples of Chern-flat metrics in higher dimensions}

In this subsection, we generalize the Chern-flat example on $Q \times \mathbb{C}$ and $M_{\lambda}$ to higher dimensions.

We write $\lambda=(\lambda_1,\lambda_2)$, where for each $1 \leq i\leq 2$, $\lambda_i=(\lambda_{1i}, \lambda_{2i})^T$ is a column vector in $\mathbb{C}^2$. Consider $\mathbb{C}^2/\Gamma$, where
\[
\Gamma=\{a\lambda_1+b\lambda_2\mid a,b\in\mathbb{Z},\lambda_1,\lambda_2\in\mathbb{C}^2\}.
\]
We always assume that $\Gamma$ is discrete in $\mathbb{C}^2$. If $\lambda_1$ and $\lambda_2$ are linearly independent over $\mathbb{C}$, then $\mathbb{C}^2/\Gamma \simeq(\mathbb{C}^*)^2$. If $\lambda_1$ and $\lambda_2$ are linearly independent over $\mathbb{R}$ but linearly dependent over $\mathbb{C}$, then $\mathbb{C}^2/\Gamma \simeq \Sigma\times\mathbb{C}$, where $\Sigma$ is an elliptic curve.

We fix the notation $G^2=G_1\times G_2$ where $G_1=G_2=\operatorname{SL}(2, \mathbb{C})$. For the sake of convenience, we write $g_j=\begin{pmatrix}a_j&b_j\\ c_j&d_j\end{pmatrix} \in G_j$ for each $1 \leq j \leq 2$.
Recall that $d(s)$ and $D$ are defined in \eqref{d(s)_def}. We consider the following subgroup of $G^2 \times \mathbb{C}^2$
\begin{equation}\label{H_lambda_high_def}
H_\lambda=H_{(\lambda_1,\lambda_2)}
=\left\{\left(d(s_1),d(s_2),-s_1\lambda_1-s_2\lambda_2\right)
\left|\ (d(s_1),d(s_2))\in D^2,s_1,s_2\right.\in\mathbb{C} \right\}.    
\end{equation}
Note that $H_{\lambda}$ is a closed subgroup of $G^2\times\mathbb{C}^2$. Now the complex manifold of our interest is 
\begin{equation}\label{M_lambda_high_def}
M_{\lambda}=H_\lambda\backslash(G^2\times\mathbb{C}^2),
\end{equation} 
with the equivalence relation
\(
(g_1,g_2,u)\sim \left(d(s_1)g_1,d(s_2)g_2,u-s_1\lambda_1-s_2\lambda_2\right).
\)
\begin{theorem}\label{M_lambda_example}
$M_\lambda$ is a simply connected complex manifold which admits a complete Chern-flat metric.
\end{theorem}

\begin{proof}[Proof of Theorem \ref{M_lambda_example}]
 $H_\lambda$ is a connected closed complex subgroup of $G^2\times\mathbb{C}^2$, since it is the image of $\mathbb{C}^2$ under the holomorphic map in \eqref{H_lambda_high_def}. As $G^2\times\mathbb{C}^2$ is simply connected, the long exact sequence of homotopy groups gives $\pi_1(M_\lambda)=0$; see \cite[Lemma 4.2]{MWY2026} for a similar argument. It suffices to construct a global holomorphic frame and prove the completeness of the corresponding Hermitian metric.

For simplicity, we use the notation from the previous diagram \eqref{diagram_maps}.
\begin{equation}\label{diagram_maps_high}
\begin{tikzcd}[column sep=40pt, row sep=30pt]
G^2 \times \mathbb{C}^2
    \arrow[r, "q"]
    \arrow[d, "\operatorname{pr}_1"']
  & M_{\lambda}
    \arrow[d, "\rho"] \\
   G^2
     \arrow[r, "{(\pi, \pi)}"]
  & Q^2
\end{tikzcd}
\end{equation}
Here $q:G^2\times\mathbb{C}^2\rightarrow M_\lambda$ is the quotient map. The projection to $G^2$ induces a holomorphic $(\mathbb{C}^2/\Gamma)$-bundle map
\begin{equation}\label{rho_high_def}
\rho:M_\lambda\longrightarrow Q^2,\qquad
\rho([g_1,g_2,u])=(\pi(g_1),\pi(g_2)),\ \
\end{equation}
where $\pi: G \rightarrow Q$ is defined in \eqref{Z_formula_1}. For $j=1,2$, we write
\[
q_j\doteq \pi(g_j)=\begin{pmatrix}x_j&y_j\\z_j&-x_j\end{pmatrix}\in Q_j,
\ \text{where}\ x_j^2+y_jz_j=1.
\]
Then we can follow \eqref{local_trivial_1} and \eqref{local_trivial_2} to write down the local trivializations of $\rho$. For example,  we choose $u=(u_1,u_2)^{T}$ as the standard coordinates on $\mathbb{C}^2$ and set
\[
U_{1+}\doteq\{x_1 \neq -1\ |\ q_1 \in Q\},\ \text{and}\ \ U_{2+}\doteq\{x_2 \neq -1\ |\ q_2 \in Q\}.
\]
Then we have
\begin{equation}\label{local_tri_high}
\rho^{-1}(U_{1+} \times U_{2+})=\Bigl[\Bigl(\begin{pmatrix}
    1  &  \frac{y_1}{1+x_1} \\ 
    -\frac{z_1}{2}  &  \frac{1+x_1}{2}
\end{pmatrix}, \begin{pmatrix}
    1  &  \frac{y_2}{1+x_2} \\ 
    -\frac{z_2}{2}  &  \frac{1+x_2}{2}
\end{pmatrix}\Bigr), u_{(1+,2+)} \doteq u+\frac{\lambda_1}{2\pi i} \log a_1+\frac{\lambda_2}{2\pi i} \log a_2 \Bigr].    
\end{equation} Here the local fiber coordinate $u_{(1+,2+)}$ takes value in $\mathbb{C}^2/\Gamma$. 

Motivated by \eqref{eta_pull_back}, we define a $\mathbb{C}^2$-valued holomorphic $1$-form $\widehat{\eta}=(\widehat{\eta}^1,\widehat{\eta}^2)^T$ on $M_{\lambda}$ such that its pull-back on $G^2 \times \mathbb{C}^2$ satisfies
\begin{equation}\label{eta_high_def}
q^{\ast}(\widehat{\eta}^a)
=du_a+\frac{1}{4\pi i}\sum_{j=1}^2\lambda_{aj}
\operatorname{tr}\bigl(h(dg_j)g_j^{-1}\bigr),
\qquad 1\leq a\leq 2.
\end{equation}
We see that \eqref{eta_high_def} is well defined on $M_{\lambda}$. First of all, the right-hand side of \eqref{eta_high_def} is (left) $H_\lambda$-invariant. Moreover, for any $j=1, 2$, the tangent vector to the $j$th parameter direction of an $H_\lambda$-orbit has $G_j$-component $2\pi i hg_j$ and $\mathbb{C}^2$-component $-\lambda_j$. The value of the right-hand side of \eqref{eta_high_def} on this vector is
\(
-\lambda_{aj}+\frac{\lambda_{aj}}{4\pi i} \operatorname{tr}(2\pi i h^2)=0.
\)
Alternatively, we may write down the local expressions of $\widehat{\eta}$ on $M_{\lambda}$ explicitly as in \eqref{eta_two_locals} and show they are independent of local trivializations. The coordinate vector fields $\frac{\partial}{\partial u_1}$ and $\frac{\partial}{\partial u_2}$ descend to global holomorphic vector fields on $M_\lambda$. They satisfy
\(
\widehat{\eta}^a\left(\frac{\partial}{\partial u_b}\right)=\delta_{ab},
\) for $1\leq a,b\leq 2$.
Thus $\widehat{\eta}$ defines a holomorphic connection with the horizontal distribution $\operatorname{ker}(\widehat{\eta})$. For any $V\in T^{1,0}Q^2$, we define its horizontal lift as in \eqref{horizontal_lift_1}.
\begin{equation}\label{horizontal_lift_high}
(V)^H=(V)-\sum_{a=1}^2
\widehat{\eta}^a\bigl((V)\bigr)\frac{\partial}{\partial u_a},
\end{equation}
where $(V)$ denotes a local coordinate lift of $V$. One may check that \eqref{horizontal_lift_high} is independent of the local coordinate lift, as it is uniquely characterized by $d\rho((V)^H)=V$ and $\widehat{\eta}((V)^H)=0$.

Let $Z_{1,j},Z_{2,j},Z_{3,j}$ be the vector fields in \eqref{Z_formula} on the $j$th factor. As in \eqref{frame_E_i}, we define
\begin{equation}\label{frame_high_horizontal}
\begin{split}
E_1&=(Z_{1,1})^H+\frac{x_1}{\sqrt2}\frac{\partial}{\partial u_1},\qquad
E_2=(Z_{2,1})^H+\frac{z_1}{2}\frac{\partial}{\partial u_1},\qquad
E_3=(Z_{3,1})^H+\frac{y_1}{2}\frac{\partial}{\partial u_1},\\
E_4&=(Z_{1,2})^H+\frac{x_2}{\sqrt2}\frac{\partial}{\partial u_2},\qquad
E_5=(Z_{2,2})^H+\frac{z_2}{2}\frac{\partial}{\partial u_2},\qquad
E_6=(Z_{3,2})^H+\frac{y_2}{2}\frac{\partial}{\partial u_2}.
\end{split}
\end{equation}
We verify that they form a global holomorphic frame on $M_{\lambda}$. In view of \eqref{Z_formula_1}, at each point $q_j \in Q$, the map $\mathfrak{sl}(2,\mathbb{C}) \ni A \rightarrow [q_j,A] \in T_{q_j}Q_j$ is surjective. And its kernel is the line spanned by $q_j$. Since $q_j=\sqrt2 x_jS_1+y_jS_2+z_jS_3$, the relation
\(
\sqrt2 x_jZ_{1,j}+y_jZ_{2,j}+z_jZ_{3,j}=0
\)
generates all the linear relations among $Z_{1,j}, Z_{2,j}$, and $Z_{3,j}$. On the other hand, \eqref{frame_high_horizontal} gives
\begin{equation}\label{vertical_relation_high}
\sqrt2 x_jE_{3j-2}+y_jE_{3j-1}+z_jE_{3j}
=(x_j^2+y_jz_j)\frac{\partial}{\partial u_j}
=\frac{\partial}{\partial u_j}.
\end{equation}
If a linear combination of $E_1,\ldots,E_6$ vanishes, its image under $d\rho$ vanishes on each factor of $Q^2$. By the preceding description of the relations and \eqref{vertical_relation_high}, this linear combination is of the form $c_1\frac{\partial}{\partial u_1}+c_2\frac{\partial}{\partial u_2}$. Then $c_1=c_2=0$ because of the linear independence of $\frac{\partial}{\partial u_1}$ and $\frac{\partial}{\partial u_2}$. Therefore $E_1,\ldots,E_6$ form a global holomorphic frame on $M_{\lambda}$.

Let $\varphi^1,\ldots,\varphi^6$ be the dual holomorphic coframe of $E_1,\ldots,E_6$. Define
\begin{equation}\label{metric_high_def}
\omega_\lambda=\sqrt{-1}\sum_{k=1}^6
\varphi^k\wedge\overline{\varphi^k}.
\end{equation}
Obviously $\omega_\lambda$ is Chern-flat.

It remains to verify the completeness of $\omega_\lambda$. Consider a piecewise $C^1$ curve $\gamma:[0,t_1)\rightarrow M_\lambda$ of finite length, and write
\begin{equation}\label{frame_equation_high}
(\gamma')^{1,0}(t)
=\sum_{j=1}^2\bigl(\xi_{1,j}(t)E_{3j-2}
+\xi_{2,j}(t)E_{3j-1}+\xi_{3,j}(t)E_{3j}\bigr).
\end{equation}
According to the definition of $a(t)$ following \eqref{frame-equation}, we introduce
\[
a_j(t)=\sqrt{\sum_{k=1}^3|\xi_{k,j}(t)|^2},\qquad
a(t)=\sqrt{a_1(t)^2+a_2(t)^2}.
\]
Then
\(
L(\gamma)=\sqrt2\int_0^{t_1}a(t)dt<\infty.
\)
First we consider the projection curve $\rho(\gamma(t))=(q_1(t),q_2(t))$ on $Q^2$. Following \eqref{A(t)_f}, we introduce
\[
A_j(t)=\xi_{1,j}(t)S_1+\xi_{2,j}(t)S_2+\xi_{3,j}(t)S_3,
\qquad j=1,2.
\]
By \eqref{tangent_vec_f}, \eqref{q_u_prime_t_f}, and \eqref{frame_high_horizontal}, we have
\begin{equation}\label{base_equation_high}
q_j^{\prime}(t)=[q_j(t),A_j(t)],\qquad
\|A_j(t)\|=a_j(t).
\end{equation}
Here we use the same matrix norm as in the proof of Lemma \ref{QC_example}. It follows that
\(
\|q_j'(t)\|\leq 2\|q_j(t)\|a_j(t)
\leq 2\|q_j(t)\|a(t).
\)
Applying the Gronwall inequality as in \eqref{q(t)bound}, we find a constant $K<\infty$ such that
\begin{equation}\label{base_bound_high}
\|q_j(t)\|\leq\|q_j(0)\|
\exp\left(2\int_0^t a(\tau)d\tau\right)\leq K,
\qquad 0\leq t<t_1,\quad j=1,2.
\end{equation}
It follows that
\begin{equation}\label{q(t)_cauchy}
\|q_j(t)-q_j(s)\|\leq 2K\int_s^t a(\tau)d\tau,
\qquad 0\leq s<t<t_1.    
\end{equation}
Hence $q_j(t)$ converges to a matrix $q_j^\ast$ as $t\rightarrow t_1$. The equations defining $Q_j$ are preserved under this limit, so $q_j^\ast\in Q_j$. Therefore
\[
\rho(\gamma(t))\longrightarrow q^\ast=(q_1^\ast,q_2^\ast)\in Q^2.
\]
Next we bound the fiber coordinate. From \eqref{frame_high_horizontal} and \eqref{frame_equation_high}, we get
\[
\widehat{\eta}^j(\gamma')
=\frac{x_j}{\sqrt2}\xi_{1,j}
+\frac{z_j}{2}\xi_{2,j}+\frac{y_j}{2}\xi_{3,j},
\qquad j=1,2.
\]
The Cauchy--Schwarz inequality and \eqref{base_bound_high} give
\begin{equation}\label{eta_bound_high}
|\widehat{\eta}(\gamma')|^2
\leq\sum_{j=1}^2
\left(\frac{|x_j|^2}{2}+\frac{|y_j|^2+|z_j|^2}{4}\right)a_j(t)^2\\
=\frac14\sum_{j=1}^2\|q_j(t)\|^2a_j(t)^2
\leq\frac{K^2}{4}a(t)^2.
\end{equation}
Choose a local trivialization $\rho^{-1}(U)\simeq U\times(\mathbb{C}^2/\Gamma)$ over a neighborhood $U$ of $q^\ast$. For all $t$ sufficiently close to $t_1$, write $\gamma(t)=(q_1(t),q_2(t),u(t))$ in this trivialization. We may choose the tail of the base curve to lie in a compact subset of $U$. In view of \eqref{eta_two_locals}, we may assume that
\(
\widehat{\eta}=du+B,
\)
where $B$ is a $\mathbb{C}^2$-valued holomorphic $1$-form on $U$, pulled back to $\rho^{-1}(U)$, and $du$ denotes the invariant $1$-form on $\mathbb{C}^2/\Gamma$ induced by the standard coordinates $(u_1, u_2)$ on $\mathbb{C}^2$. 
Since $\Gamma$ is discrete, the fiber curve $u(t)$ on this tail lifts to a curve $\widetilde{u}(t)$ in $\mathbb{C}^2$. It satisfies
\[
\widetilde{u}'(t)=\widehat{\eta}(\gamma'(t))
-B\bigl(q_1'(t),q_2'(t)\bigr).
\]
The coefficients of $B$ are bounded on the compact subset containing the tail of the base curve. Combining \eqref{base_equation_high}, \eqref{base_bound_high}, and \eqref{eta_bound_high}, we obtain
\(
|\widetilde{u}'(t)|\leq K_1 a(t),
\) for some constant $K_1>0$. By a similar argument in \eqref{q(t)_cauchy}, $\widetilde{u}(t)$ converges to some $\widetilde{u}^\ast\in\mathbb{C}^2$, and
\(
\gamma(t)\longrightarrow
(q_1^\ast,q_2^\ast,[\widetilde{u}^\ast])\in M_\lambda.
\)
Thus every finite-length curve has an endpoint in $M_\lambda$. This proves that $\omega_\lambda$ is complete.
\end{proof}

In particular, if we take $\lambda_1=(1,0),\lambda_2=(\sqrt{-1},0)$. then $M_2:=M_{(\lambda_1,\lambda_2)}$ is a simply connected complete Chern-flat manifold which is not holomorphically separable, as each fiber of $M_2$ contains an elliptic curve.

By a similar argument, we can generalize Theorem \ref{M_lambda_example} to higher dimensions.

\begin{proposition}\label{M_lambda_example_Cousin}
    For any given integer $r\ge 1$ and $\lambda=(\lambda_1,\ldots,\lambda_r)$ where $ \lambda_1,\ldots,\lambda_r \in \mathbb{C}^r$ such that $\Gamma=\operatorname{span}_{\mathbb{Z}}\{\lambda\}$ is discrete in $\mathbb{C}^{r}$, $M_{r,\lambda}$ is a simply connected complex manifold which admits a complete Chern-flat metric. Here
    \[
    M_{r,\lambda}=H_{r,\lambda}\backslash(G^r\times\mathbb{C}^r)
    \]
    with
    \[
    H_{r,\lambda}=\left\{\left(d(s_1),\ldots,d(s_r),-s_1\lambda_1-\cdots-s_r\lambda_r\right)
\left|\ d(s_j)\in D, s_j\right.\in\mathbb{C},\forall j=1,2,\ldots,r \right\}
    \]
    as in \eqref{H_lambda_high_def} and \eqref{M_lambda_high_def}.
\end{proposition}

In particular, we may take $\lambda_1=(1,0,0),\lambda_2=(0,1,0), \lambda_3=(\sqrt{-1},\sqrt{-1}a,0)$ where $a \in \mathbb{R}\setminus \mathbb{Q}$. Then $M_3:= M_{3,(\lambda_1,\lambda_2,\lambda_3)}$ is a holomorphic fiber bundle over $Q^3$ with fiber $\mathcal{T} \times \mathbb{C}$. Here $\mathcal{T}$ is a noncompact complex $2$-dimensional toroidal group (also called a Cousin torus) which was first constructed by Cousin \cite{Cousin}; see also \cite{AK2001, Morimoto1965}. Note that every holomorphic function on a toroidal group is constant. For any $f \in \mathcal{O}(M_3)$ and $p \in \mathcal{T}\times \{0\}$, the holomorphic hull of $p$ contains the entire slice $\mathcal{T}\times \{0\}$, which is closed but noncompact. Hence $M_3$ is not holomorphically convex.

In view of Lemma \ref{QC_torsion}, it is possible to use \eqref{horizontal_lift_high} and \eqref{vertical_relation_high} to show that the norms of the torsions of both $M_2$ and $M_3$ also have exponential growth.

\subsection{The relationship between $M_{\lambda}$ and the Calabi-Eckmann manifold.}

In this subsection, we explore the relation between $M_\lambda$ and the Calabi--Eckmann manifold. Recall that Calabi--Eckmann manifolds \cite{CalabiEckmann} belong to a class of complex manifolds homeomorphic to products of two odd-dimensional spheres $S^{2m+1}\times S^{2n+1}$. First of all, we review the construction.

For any $\tau \in \mathbb{C}$ with $\operatorname{Im}\tau \ne 0$, we define
\begin{equation}\label{K_group_def}
K=\left\{\left(d(s),d(\tau\,s)\right) \;\middle|\; s\in\mathbb{C}\right\} \subseteq D^2,    
\end{equation}
and an action of $K$ on $(\mathbb{C}^2\setminus\{0\})^2$ by
\begin{equation}\label{K_action_def}
\left(d(s),d(\tau\,s)\right)\cdot(V_1,V_2)=\left(e^{2\pi i s}V_1,e^{2\pi i \tau s}V_2\right).
\end{equation}
Unlike in \eqref{d(s)_def} where $d(s)$ acts on $G$ by matrix multiplications, here the action is scalar multiplication in each coordinate. Let $[V_1,V_2]_K$ denote the equivalence class in the quotient. Namely, if $V_1=(V_{11}, V_{12})$, then $d(s)$ acts on $\mathbb{C}^2 \setminus \{0\}$ by $e^{2\pi i s}(V_{11}, V_{12})=(e^{2\pi i s}V_{11}, e^{2\pi i s}V_{12})$. Then $Y=K\backslash(\mathbb{C}^2\setminus\{0\})^2$ is a Calabi--Eckmann manifold. In particular, $Y$ has a holomorphic fiber bundle structure
\[
\varpi: Y \longrightarrow \mathbb{P}^1\times\mathbb{P}^1,
\qquad
\varpi\left([V_1,V_2]_K\right)=\left([V_1],[V_2]\right),
\]
where we write $V \rightarrow [V]$ as the standard quotient map from $\mathbb{C}^2 \setminus \{0\}$ to $\mathbb{P}^1$, and $(V_1, V_2) \rightarrow [V_1,V_2]_K$ as the quotient map from $(\mathbb{C}^2 \setminus \{0\})^2$ to $Y$. Note that each fiber of $\varpi$ is an elliptic curve $E_{\tau} \doteq \mathbb{C}/\operatorname{span}_{\mathbb{Z}}\{1, \tau\}$. For example, a local trivialization of $\varpi$ on $U=\{V_{11}V_{21} \neq 0\} \subset \mathbb{P}^1 \times \mathbb{P}^1$ which corresponds to $\varpi^{-1}(U)  \rightarrow U \times E_{\tau}$ can be chosen as
\[
 [(V_{11}, V_{12}), (V_{21}, V_{22})]_{K} \rightarrow [1, \frac{V_{12}}{V_{11}}] \in \mathbb{P}^1, [1, \frac{V_{22}}{V_{21}}] \in \mathbb{P}^1,\ u=\Bigl[\frac{\log V_{21}-\tau  \log V_{11} }{2\pi i}\Bigr]  \in E_{\tau}.
\]

Consider $r_1: G\to\mathbb{C}^2 \setminus \{0\}$ which is the projection map to the first row.
\begin{equation}\label{def_r_1_map}
r_1(g) = (a,b),\ \ \ \text{where}\ g=\begin{pmatrix}a&b\\c&d\end{pmatrix}.   
\end{equation}
In view of the biholomorphism $\Phi: D \backslash G \rightarrow Q$ defined in \eqref{Q_complex_homo}, we introduce
\begin{equation}\label{sigma_def}
\sigma: Q \longrightarrow \mathbb{P}^1, \qquad  \sigma\left(\Phi(D\begin{pmatrix}a&b\\c&d\end{pmatrix})\right)=[r_1\left(\begin{pmatrix}a&b\\c&d\end{pmatrix}\right)]=[a: b].    
\end{equation}
Recall that $M_\lambda=H_\lambda\backslash(G^2\times\mathbb{C}^2)$, where $\lambda_1=(1,0)^T$, $\lambda_2=(-\frac{1}{\tau}, 0)^T$.  We summarize these maps in the following diagram.
\begin{equation}\label{diagram_maps_high_2}
\begin{tikzcd}[column sep=45pt, row sep=30pt]
G^2 \times \mathbb{C}^2
    \arrow[r, "q"]
   & M_{\lambda}\simeq N_\lambda\times\mathbb{C}
    \arrow[d, "\rho"] 
    & (\mathbb{C}^2\setminus\{0\})^2
    \arrow[r, "{[\cdot, \cdot]_K}"]
    \arrow[d, "{([\cdot], [\cdot])}"]
    & Y
    \arrow[d, "\varpi"]
    \\
   &  Q^2
      \arrow[r, "{(\sigma, \sigma)}"]
      &
      \mathbb{P}^1 \times  \mathbb{P}^1 
     \arrow[r, "{(\operatorname{id}, \operatorname{id})}"]
  & \mathbb{P}^1 \times  \mathbb{P}^1 
\end{tikzcd}
\end{equation}
Note that $M_\lambda\simeq N_\lambda\times\mathbb{C}$, where $N_\lambda \simeq H_\lambda \backslash(G^2\times\mathbb{C})$ is with the equivalence relation
\(
(g_1,g_2,u)\sim \left(d(s_1)g_1,d(s_2)g_2,u-s_1+\frac{1}{\tau}\,s_2\right).
\)
According to \eqref{local_tri_high}, each fiber of $\rho: M_{\lambda} \rightarrow Q^2$ is the elliptic curve $\mathbb{C}/\operatorname{span}_{\mathbb{Z}}\{1, -\frac{1}{\tau}\}$. It is naturally identified with $E_{\tau}$, which is the fiber of $\varpi: Y \rightarrow \mathbb{P}^1 \times \mathbb{P}^1$. As $(\sigma, \sigma): Q^2 \rightarrow \mathbb{P}^1 \times \mathbb{P}^1$, we are led to consider the pullback bundle $(\sigma, \sigma)^*Y$.
\begin{equation}
(\sigma, \sigma)^*Y=\left\{(q_1,q_2,\eta)\in Q^2 \times Y \ \ |\ \ \left(\sigma(q_1),\sigma(q_2)\right)=\varpi(\eta)\right\}.    
\end{equation}
Then we observe the following result.
\begin{lemma}\label{CE_M_lambda}
$N_\lambda$ is biholomorphic to the total space of $(\sigma, \sigma)^*Y$. Therefore $M_\lambda$ is biholomorphic to the product of $\mathbb{C}$ and the total space of $(\sigma, \sigma)^*Y$.     
\end{lemma}

\begin{proof}[Proof of Lemma \ref{CE_M_lambda}]

We proceed to show that both $N_\lambda$ and $(\sigma, \sigma)^*Y$ are biholomorphic to $K\backslash G^2$. First of all, we define the holomorphic map
\[
\Psi:N_\lambda\longrightarrow K\backslash G^2,
\qquad
\Psi\left([g_1,g_2,u]\right)=[d(u)g_1,g_2]_{K \subset D^2}.
\]
Here $[g_1,g_2]_{K \subset D^2}$ denotes the equivalence class obtained by regarding $K$ as a subgroup of $D^2$. We may check that $\Psi$ is well-defined. To see it, we note that the  $(d(u)g_1, g_2)$ and $(d(u-s_1+\frac{s_2}{\tau})d(s_1)g_1, d(s_2)g_2)$ are related by a left action of $K$ on $G^2$. Moreover, $\Psi$ admits an inverse
\(
\Psi^{-1}\left([g_1,g_2]_{K \subset D^2}\right)=[g_1,g_2,0].
\)
Recall that $r_1$ is defined in \eqref{def_r_1_map}. We introduce $\widehat{\Psi}:K\backslash G^2 \longrightarrow (\sigma, \sigma)^*Y$ by
\begin{equation}\label{def_Psi_hat}
\widehat\Psi\left([g_1,g_2]_{K \subset D^2}\right)=\left(\Phi(Dg_1), \Phi(Dg_2), [r_1(g_1),r_1(g_2)]_K\right).
\end{equation}
Since
\(
r_{1}(d(s)g_1)=e^{2\pi i s}r_1(g_1)
\)
and
$r_{1}(d(\tau\,s)g_2)=e^{2\pi i \tau s}r_1(g_2)$, we have
\[
[r_1(g_1),r_1(g_2)]_K=\left[e^{2\pi i s}r_1(g_1),e^{2\pi i \tau s}r_1(g_2)\right]_K.
\]
Therefore, $\widehat\Psi$ is a well-defined holomorphic map.

It remains to verify that $\widehat\Psi$ is invertible. Take any $(q_1,q_2,\eta) \in (\sigma, \sigma)^*Y$. One can choose $g_1,g_2\in G$ such that $q_j=\Phi(Dg_j)$ and a pair $(V_1,V_2)\in(\mathbb{C}^2\setminus\{0\})^2$ such that $\eta=[V_1,V_2]_K$. It follows that
\begin{equation*}
([V_1],[V_2])=\varpi(\eta)
=(\sigma(q_1),\sigma(q_2))
=([r_1(g_1)],[r_1(g_2)]).    
\end{equation*}
Hence $[V_j]=[r_{1}(g_j)] \in \mathbb{P}^1$ and there exist $c_j\in\mathbb{C}^*$ such that $V_j=c_j\,r_1(g_j)$ for $j=1,2$.
Define
\begin{equation}\label{def_Psi_inv}
\widehat\Psi^{-1}(q_1,q_2,\eta)=\left[
\begin{pmatrix}c_1&0\\0&c_1^{-1}\end{pmatrix}g_1,
\begin{pmatrix}c_2&0\\0&c_2^{-1}\end{pmatrix}g_2
\right]_{K \subset D^2}.
\end{equation}
Now we verify that this definition of $\widehat\Psi^{-1}$ is independent of the choice of $g_1,g_2$ and $(V_1,V_2)$. Suppose that there exist $g_1' , g_2', (V_1',V_2')$ such that $q_j=\Phi(Dg_j')$ for $j=1,2$ and $\eta=[V_1',V_2']_K$. For simplicity, we write $\delta(t) \doteq \begin{pmatrix}
    t&0 \\0 & t^{-1}
\end{pmatrix}$. Since $\Phi:D\backslash G\to Q$ is biholomorphic, there exist $t_1,t_2$ such that $g'_j=\delta(t_j) g_j$ for $j=1,2$. Now $\eta=[V_1,V_2]_K=[V_1',V_2']_K$. Then there exists $s\in \mathbb{C}$ such that $V'_1=e^{2\pi i s}V_1$ and $V'_2=e^{2\pi i\tau s}V_2$. It follows from $V_j'=c_j' r_1(g_j')$ that
\[
c_1'= e^{2\pi i s} c_1 \frac{1}{t_1}, \qquad c_2'= e^{2\pi i \tau s} c_2 \frac{1}{t_2}.
\]
Therefore
\[
\delta(c'_j)g'_j
=\delta\left(\frac{\alpha_jc_j}{t_j}\right)\delta(t_j)g_j
=\delta(\alpha_j)\delta(c_j)g_j,
\]
where $\alpha_1= e^{2\pi i s}$, $\alpha_2=e^{2\pi i \tau s}$. Hence $(\delta(c'_1)g'_1,\delta(c'_2)g'_2)$ and 
$(\delta(c_1)g_1,\delta(c_2)g_2)$ are related by an action of $(d(s),d(\tau s)) \in K$. We have shown that the right hand side of \eqref{def_Psi_inv} is well-defined. One can directly check that $\widehat\Psi$ is biholomorphic.
\end{proof}

\appendix

\section{An alternative approach to Theorem \ref{thm:high-dimension} via Palais' theorem}\label{app_A}

In this section we provide an alternative proof of the following special case of Theorem \ref{thm:high-dimension}.

\begin{theorem}\label{thm:high-dimension_part}
    Let $(M^n,g)$ be a complete simply connected Chern-flat Hermitian manifold which satisfies Property (H). Then $M$ is biholomorphic to $\mathbb{C}^n$. 
\end{theorem}

\begin{lemma}\label{Palais}
    Let $\mathfrak{g}$ be a finite-dimensional complex Lie algebra of holomorphic
vector fields on a complex manifold $M$, generated by $E_1,\ldots, E_m \in H^0(M, T^{1, 0}M)$. Let $G$ be the simply connected complex Lie group associated with the Lie algebra $\mathfrak{g}$. 
Assume that each $E_j$ is complete. Then every $Z \in\mathfrak{g}$ generates a complete holomorphic flow on $M$. Moreover, there exists a holomorphic left action $L: G\times M\to M$ such that
\begin{equation}\label{eq:left-infinitesimal-sign}
 Z(x)=\left.\frac{d}{dt}\right|_{t=0}L(\exp(-tZ),x),
 \qquad \forall\,Z\in\mathfrak{g}.
\end{equation}

\end{lemma}

Lemma \ref{Palais} is a holomorphic version of an important result due to Palais \cite[Theorem III, p.~95]{Palais1957}; see \cite[Corollary 5.1]{Jouan2010} for an alternative proof of Palais' theorem. We observe that the same conclusion remains valid for holomorphic vector fields. For the sake of convenience, we include a proof of Lemma \ref{Palais} based on the exposition of a related result in \cite[Theorem~6.5]{Michor}.

\begin{proof}[Proof of Lemma \ref{Palais}]
We divide the proof into the following steps.

\textbf{Step 1:} Construct a complete basis of $\mathfrak{g}$.

Let $X \in \mathfrak{g}$ be complete which generates a holomorphic flow $\Phi_t^X$. As in the proof of Theorem \ref{thm:surface-rigidity}, set $Y_t=(\Phi_{t}^X)^*(Y) = (\Phi_{-t}^X)_*Y$. Note that $Y_t$ is complete if and only if $Y$ is complete. Since $\mathfrak{g}$ is generated by complete holomorphic vector fields $E_1,\ldots,E_m$, $(\Phi_{t_i}^{E_i})^* E_j$ is complete. In general, let $\mathcal{C}$ consist of the generators and all fields obtained by
finitely many operations $Y\mapsto(\Phi_t^{E_i})^* Y$. Then each element in $\mathcal{C}$ is complex-complete.  Let $\mathfrak{a}=\Span_{\mathbb{C}}\{Y \in \mathcal{C}\}$. Then $\mathfrak{a}$ is a subspace of $\mathfrak{g}$, which is invariant under
$\Phi_t^{E_i}$. For any $Y \in \mathfrak{a}$, we have $(\Phi_t^{E_i})^* Y\in \mathfrak{a}$. Hence $[E_i, Y] \in \mathfrak{a}$ by taking Lie derivatives. That is to say, $[E_i, \mathfrak{a}] \subset \mathfrak{a}$. However, $\mathfrak{g}$ is generated by such $E_i$. Hence $\mathfrak{g}=\mathfrak{a}$ and we may choose a complete basis in $\mathcal{C}$.

\textbf{Step 2:} Construct a holomorphic involutive distribution. 

For any given $Z \in \mathfrak{g}$, let $Z^R$ denote the right-invariant holomorphic vector field on $G$. Define
\begin{equation}\label{complex-distribution}
 \mathcal{D}_{(g,x)}
 =\{(Z^R(g),-Z(x)) \ | \ Z\in\mathfrak{g}\}
 \subset T^{1,0}_{(g,x)}(G\times M).
\end{equation}
Recall that for right-invariant vector fields,  $[Z^R,W^R]=-[Z,W]^R$. Suppose $(Z^R,-Z)$, $(W^R,-W) \in \mathcal{D}$. Then
\[
 [(Z^R,-Z),(W^R,-W)]
 =(-[Z,W]^R,[Z,W])=-([Z,W]^R,-[Z,W]) \in \mathcal{D}.
\]
Therefore $\mathcal{D}$ is involutive. By the holomorphic Frobenius theorem, $\mathcal{D}$ is integrable and there exists a maximal connected integral leaf $\mathcal{L}_x$ through $(\mathbf{1}_G,x)$. Let $\operatorname{pr}_1:G\times M \to G$ be the natural projection to the first factor. The restricted
projection
\[
\pi_x= \operatorname{pr}_1|_{\mathcal{L}_x}: \mathcal{L}_x \to G
\]
is a local biholomorphism because $(\pi_x)_* (Z^R(g),-Z)=Z^R(g)$. 

\textbf{Step 3:} We show that $\pi_x$ is surjective.

Choose $\{Z_1,Z_2,\ldots,Z_s\} \subset \mathcal{C}$ to be a complete basis of $\mathfrak{g}$. For simplicity of notation, write $\widehat{Z}_j=(Z_j^R,-Z_j)$. Then each $\widehat{Z}_j$ has complete complex flow
\begin{equation}\label{productflow}
 \widehat{\Phi}_t^j(g,y)
=\bigl(\exp(tZ_j)g,\Phi_{-t}^{Z_j}(y)\bigr).
\end{equation}
These flows preserve each maximal leaf.
For $z=(z_1,\ldots,z_s) \in \mathbb{C}^s$, set
\[
\Xi(z)=\exp(z_s Z_s) \cdots\exp(z_1 Z_1).
\]
Since $d\Xi_0(w)=\sum_j w_j Z_j$, $\Xi$ is a local biholomorphism near the origin. Choose $B=B(0,\varepsilon)$ sufficiently small on which $\Xi$ is biholomorphic to an open neighborhood $U$ of  $\mathbf{1}_G$.

We claim a basic fact: if $g \in \pi_x(\mathcal{L}_x)$, then $Ug \subset \pi_x(\mathcal{L}_x)$. Indeed, if $g \in \pi_x(\mathcal{L}_x)$, then there exists $q \in \mathcal{L}_x$ such that $\pi_x(q)=g$. For any $u \in U$, $u=\Xi(z)$ for some $z=(z_1,\ldots,z_s) \in B$. Consider 
\[
q_z=\widehat{\Phi}_{z_s}^s \circ \cdots \circ \widehat{\Phi}_{z_1}^1 (q) \in \mathcal{L}_x.
\]
It follows from \eqref{productflow} that
\[
\pi_x (q_z)=\Xi(z)g=ug.
\]
Consequently
\[
 Ug\subset \pi_x(\mathcal L_x).
\]

Now we prove $\pi_x(\mathcal{L}_x)=G$. First, $\mathbf{1}_G \in \pi_x(\mathcal{L}_x)$ since $(\mathbf{1}_G,x) \in \mathcal{L}_x$. 
Take an open symmetric neighborhood $V=V^{-1}\subset U$ centered at $\mathbf{1}_G$, and define
\[
 V^k=\{v_1\cdots v_k\ | \ v_1,\ldots,v_k\in V\},\qquad k\geq1.
\]
For any $g \in V$, $Vg \subset Ug \subset \pi_x(\mathcal{L}_x)$ by the previous claim. Hence $V^k \subset \pi_x(\mathcal{L}_x)$ by induction.
It's direct to check that the union $K=\bigcup_{k\geq1}V^k$ is a subgroup of $G$. Every $V^k$ is open, so $K$ is open. Connectedness of $G$ implies that $K=\pi_x(\mathcal{L}_x)=G$. Thus $\pi_x$ is surjective.

\textbf{Step 4:} We prove that $\pi_x$ is a biholomorphism.

Since $G$ is simply connected, it suffices to show that $\pi_x:\mathcal L_x\longrightarrow G$ is a covering map. We claim that for any $g_0\in G$, there exists a neighborhood $\mathcal{U}_y$ where $y \in \pi^{-1}_{x}(g_0)$, such that 
\[
\pi^{-1}_x(Ug_0)=\bigsqcup_{y\in\pi^{-1}_{x}(g_0)}\mathcal{U}_y
\]
and $\pi_x:\mathcal{U}_y\to Ug_0$ is biholomorphic. 
Fix $g_0\in G$ and set
\[
Ug_0=\{ug_0\mid u\in U\},
\]
where $U=\Xi(B)$ is the neighborhood constructed in Step~3. For each $y\in\pi_x^{-1}(g_0)$, define $S_y:Ug_0\longrightarrow\mathcal L_x$ by
\[
S_y(ug_0)=\widehat{\Phi}_{z_s}^{\,s}\circ\cdots\circ\widehat{\Phi}_{z_1}^{\,1}(y), \qquad u=\Xi(z).
\]
Since $\Xi:B\to U$ is biholomorphic, the map $S_y$ is well defined and holomorphic. By \eqref{productflow},
\begin{equation}\label{project}
    \pi_x\circ S_y=\operatorname{id}_{Ug_0}.
\end{equation}
Since the tangent map is $(\pi_x)_*\circ (S_y)_*=\operatorname{id}$, we obtain $(S_y)_*$ is an isomorphism. Using the inverse function theorem, $S_y$ is locally biholomorphic. Hence we just pick $U$ sufficiently small and  $\mathcal{U}_y=S_y(Ug_0)$.

We verify that $S_y(Ug_0)$ are pairwise disjoint for $y \in \pi^{-1}_{x}(g_0)$. Assume that $S_y(Ug_0) \cap S_{y'}(Ug_0)$ is nonempty for some $y,y' \in \pi^{-1}_{x}(g_0)$. Then there exist $u,u' \in U$, such that $S_y(ug_0)=S_{y'}(u' g_0)$. It follows from \eqref{project} that $u=u'= \Xi(z)$ for some $z\in B$. Write $T_z=\widehat{\Phi}_{z_s}^s \circ \cdots \circ \widehat{\Phi}_{z_1}^1$, where all these flows are diffeomorphisms. Then $T_z(y)=T_z(y')$, which implies that $y=y'$. 

It remains to show 
\[
\pi^{-1}_x(Ug_0)=\bigcup_{y\in\pi^{-1}_{x}(g_0)} S_y(Ug_0).
\]
Take $a \in \pi^{-1}_{x}(Ug_0)$ and then $\pi_x (a) \in Ug_0$. It follows that there exists $u= \Xi(z)$ such that $\pi_x(a)=\Xi(z)g_0$. Pick $b=T_z^{-1}(a)$. Then $b \in \mathcal{L}_x$ since the flow preserves the integral leaf. Moreover, $\pi_x(b)=g_0$. Hence
\[
a=T_z(b) \in S_b(Ug_0) \subset \bigcup_{y\in\pi^{-1}_{x}(g_0)} S_y(Ug_0).
\]

\textbf{Step 5:} Construct the desired holomorphic left action.

Define
\[
L: G\times M \to M, \qquad (g,L(g,x))=\pi_x^{-1}(g).
\]
Then $L$ is holomorphic since the projection maps are both holomorphic. We now verify that $L$ gives a left action.

For any fixed $h \in G$, define
\[
\mathcal{R}_h: G\times M \to G\times M, \qquad \mathcal{R}_h(g,y)=(gh,y).
\]
Note that $\mathcal{R}_h$ preserves $\mathcal{D}$ since $\mathcal{D}$ is defined by the right-invariant vector fields. Denote $q=L(h,x)$, and then $(h,q) \in \mathcal{L}_x$. On the other hand, $\mathcal{R}_h(\mathcal{L}_q)$ is the maximal integral leaf through $(h,q)$ since $\mathcal{R}_h(\mathbf{1}_G,q)=(h,q)$. It follows that $\mathcal{R}_h(\mathcal{L}_q)=\mathcal{L}_x$. Pick $(g,L(g,q)) \in \mathcal{L}_q$. We have
\[
\mathcal{R}_h(g,L(g,q))=(gh,L(g,q)) \in \mathcal{L}_x.
\]
Hence
\[
L(gh,x)=L(g,q)=L(g,L(h,x)).
\]
We conclude that $g\cdot x:=L(g,x)$ is a left action. Note that $\mathcal{D}_{(\mathbf{1}_G,x)}=\{(Z,-Z(x))\ | \ \forall Z\in \mathfrak{g}\}$ by \eqref{complex-distribution}. Hence
\[
\left.\frac{d}{dt}\right|_{t=0}L(\exp(tZ),x)= L_{* (\mathbf{1}_G,x)}\left(\left.\frac{d}{dt}\right|_{t=0}\exp(tZ),0\right) =L_{* (\mathbf{1}_G,x)}(Z,0) =-Z(x).
\]
    
\end{proof}

\begin{proof}[Proof of Theorem \ref{thm:high-dimension_part}]
    Let $(M^n,g)$ be a simply connected complete Chern-flat Hermitian manifold with $\{E_1,\ldots,E_n\}$ a global holomorphic frame defined in \eqref{eq:parallel-frame}. If there are holomorphic functions with polynomial growth which give local coordinates near a point, then $\{E_1,\ldots,E_n\}$ generates a nilpotent Lie algebra $\mathfrak{g}$ with finite dimensions by Proposition \ref{prop:nilpotent-brackets}. Let $G$ be the simply connected nilpotent complex Lie group with its Lie algebra $\mathfrak{g}$. Then Lemma \ref{Palais} gives a left action $L: G\times M \to M$. Fix $x \in M$ and consider the orbit
    \[
    G_x=\{L(g,x)\ | \ g\in G \}.
    \]
    Denote 
    \[
    F_x:G \to M, \qquad F_x(g)=L(g,x).
    \]
    Then 
    \[
    (F_x)_{*\mathbf{1}_G}(E_j)=-E_j(x), \qquad \forall j=1,2,\ldots,n.
    \]
    Hence $(F_x)_{*\mathbf{1}_G}$ is surjective. It follows from the implicit function theorem that there exists a neighborhood of $\mathbf{1}_G$ in $G$, whose image under $F_x$ contains an open neighborhood $U_x$ centered at $x$. That is to say, $x \in U_x \subset G_x$. The above argument holds for any $y \in G_x$. Hence 
    \[
    G_x=\bigcup_{y\in G_x} U_y, 
    \]
    which is open in $M$. Thus $G_x=M$ and $M$ is actually a homogeneous manifold. By the standard theory of homogeneous complex manifolds, see \cite[p. 8]{Akhiezer_B} for example, $M$ is biholomorphic to $G/H$, where $H$ is a closed subgroup of $G$. Moreover, the holomorphic fibration $H \to G \to G/H$ induces a homotopy exact sequence
    \[
    \to \pi_1(G) \to \pi_1(M) \to \pi_0(H) \to \pi_0(G) \to.
    \]
    Then $H$ is connected since both $\pi_1(M)$ and $\pi_0(G)$ are trivial. We conclude that $M \cong G/H$ is biholomorphic to $\mathbb{C}^n$ by a result of Matsushima \cite[Lemma 3.1]{Ma1961}.
\end{proof}

\begin{example}\label{ex:group_act}
We return to the example of a Chern-flat Hermitian metric discussed in Remark \ref{rem:no-isometric-conclusion}. Let $(M, g)$ be a complete Chern-flat metric on $\mathbb{C}^3$ defined by the dual frame \eqref{eta_xi}. By \cite[Proposition 5.4]{MWY2026}, it satisfies Property (H). In order to illustrate the proof of Theorem \ref{thm:high-dimension_part}, we show that $M$ can be realized as a complex homogeneous space $H \backslash G$ in an explicit way. Here we adopt the notation of the right action of $G$ on $M$ as it meshes nicely with the left-invariant frame on $G$. Of course, it is a matter of notation to rewrite the proofs of Lemma \ref{Palais} and Theorem \ref{thm:high-dimension_part} using a right action.  
\end{example}

First we solve the corresponding global unitary holomorphic frame on $M$ as follows. 
\begin{equation}
  E_1=\frac{\partial}{\partial z_1}+\eta \frac{\partial}{\partial z_3},\ \ \  E_2=\frac{\partial}{\partial z_2}+\xi \frac{\partial}{\partial z_3}, \ \ \ E_3=\frac{\partial}{\partial z_3}. 
  \label{f_global_frame}    
\end{equation}
Here we choose $\eta=0$ and $\xi=\frac{1}{6}z_1^3$. Then the complex Lie algebra $\mathfrak{g}$ generated by $E_1, E_2$, and $E_3$ is a $5$-dimensional filiform algebra of the first kind in the notation of \cite[Example 9.1.7]{LeDonne_B}. Its basis can be expressed as
\begin{equation}
  X_1=\frac{\partial}{\partial z_1},\ \ \  X_2=\frac{\partial}{\partial z_2}+\frac{1}{6}z_1^3 \frac{\partial}{\partial z_3}, \ \ \ X_3=\frac{1}{2}z_1^2\frac{\partial}{\partial z_3},\ \ \  X_4=z_1\frac{\partial}{\partial z_3},\ \ \ X_5=\frac{\partial}{\partial z_3}, 
  \label{f_frame_basis}    
\end{equation} with the only nontrivial relations as follows
\[
[X_1, X_2]=X_3,\ \ \ \  [X_1, X_3]=X_4,\ \ \ \  [X_1, X_4]=X_5
\] besides the corresponding commutation ones. In other words, $\operatorname{ad}_{X_1}$ acts on $\operatorname{span}\{X_2, X_3, X_4, X_5\}$ which is a $4$-dimensional abelian ideal in $\mathfrak{g}$ by
\[
\operatorname{ad}_{X_1}(X_i)=X_{i+1},\ \ \text{for}\ 2 \leq i \leq 4,\ \ \ \ \ \ \operatorname{ad}_{X_1}(X_5)=0.
\]
Therefore, the simply connected complex Lie group $G$ associated with $\mathfrak{g}$ can be written as $\mathbb{C} \ltimes_{\varphi} \mathbb{C}^4$, with the group multiplication defined by
\begin{align*}
&(a, b, c, d, e) \cdot (\widetilde{a}, \widetilde{b}, \widetilde{c}, \widetilde{d}, \widetilde{e})
=\Bigl(a+\widetilde{a}, (b, c, d, e)+\varphi(a)(\widetilde{b}, \widetilde{c}, \widetilde{d}, \widetilde{e})\Bigr),\ \ \text{where}\\
& \varphi(a)(\widetilde{b}, \widetilde{c}, \widetilde{d}, \widetilde{e})=(\widetilde{b}, \widetilde{c}+a\widetilde{b}, \widetilde{d}+a\widetilde{c}+\frac{a^2}{2}\widetilde{b}, \widetilde{e}+a\widetilde{d}+\frac{a^2}{2}\widetilde{c}+\frac{a^3}{6}\widetilde{b}).
\end{align*}

Now we define a natural $G$-action on $M$ by
\begin{equation}
(z_1, z_2, z_3) \cdot (a, b, c, d, e)=(z_1+a, z_2+b, z_3+\frac{z_1^3}{6}b+\frac{z_1^2}{2}c+dz_1+e).
\label{G_right_action}    
\end{equation} We may check that it is a right action with the infinitesimal generators being exactly \eqref{f_frame_basis}. For example, we solve
\(
\left.\frac{d}{dt}\right|_{t=0} (z_1, z_2, z_3) \cdot (0, 0, t, 0, 0)=X_3.
\)
Moreover, $G$ acts transitively on $M$. Let $H$ be the stabilizer group at $(0, 0, 0)$. Then \(H=\{(0, 0, c, d, 0) \in G\ |\ c, d \in \mathbb{C}\}\) is the abelian group $\mathbb{C}^2$. The corresponding Lie algebra $\mathfrak{h}$ is $\operatorname{span}\{X_3, X_4\}$. We get that $M$ is biholomorphic to $H \backslash G$.  Moreover, in this particular example, we may express the correspondence more explicitly in the spirit of \cite[Lemma 3.1]{Ma1961}. Note that every element $g=(a, b, c, d, e) \in G$ can be uniquely expressed as
\begin{equation}
    g=h\,\exp(eX_5)\,\exp(bX_2)\,\exp(aX_1),\ \ \text{where}\ h=\exp(cX_3+dX_4)=(0, 0, c, d, 0) \in H.  \label{g_formula}
\end{equation}
Therefore, we have an explicit biholomorphic map $F: H \backslash G \rightarrow \mathbb{C}^3$, defined by
\[
(z_1, z_2, z_3)=F([g])=(a, b, e)\ \ \ \text{where}\ g \text{ is defined in }\eqref{g_formula} \text{ and }[g] \in H \backslash G.
\]

\bibliographystyle{amsplain}

\bibliography{chern_flat}

\end{document}